\documentclass[11pt]{article}
\usepackage{amsmath,amsthm,amsfonts,amssymb,amscd}
\usepackage{bbm}
\numberwithin{equation}{section}
\usepackage{enumerate}
\usepackage{fancyhdr}
\usepackage{color}
\usepackage{float}
\usepackage{soul}
\usepackage{graphicx}
\usepackage[margin=1in]{geometry}
\theoremstyle{plain}
\newtheorem{theorem}{Theorem}[section]
\newtheorem{lemma}[theorem]{Lemma}
\newtheorem{corollary}[theorem]{Corollary}
\newtheorem{proposition}[theorem]{Proposition}
\newtheorem{definition}{Definition}[section]

\theoremstyle{definition}

\def\disp{\displaystyle}

\newcommand{\la}{\langle}
\newcommand{\ra}{\rangle}
\newcommand{\ve}{\varepsilon}

\newcommand{\gph}{{\rm gph}\,}

\def\tto{\;{\lower 1pt\hbox{$\rightarrow$}}\kern-10pt
\hbox{\raise 2pt\hbox{$\rightarrow$}}\;}
\newcommand{\B}{\mathbbm{B}}
\newcommand{\NN}{\mathbbm{N}}

\newcommand{\os}{\bar s}
\newcommand{\gm}{\gamma}
\newcommand{\dom}{{\rm dom\,}}

\def\dn{\downarrow}

\newcommand{\al}{\alpha}
\newcommand{\prox}{{\rm prox}}
\newcommand{\nexto}{\kern -0.54em}

\newcommand{\dZ}{{\cal Z \kern -0.7em Z}}
\newcommand{\dC}{{\rm\hbox{C \kern-0.8em\raise0.2ex\hbox{\vrule height5.4pt width0.7pt}}}}
\newcommand{\dQ}{{\rm\hbox{Q \kern-0.85em\raise0.25ex\hbox{\vrule height5.4pt width0.7pt}}}}

\newcommand{\Id}{{\rm Id}\,}

\newenvironment{retraitsimple}{\begin{list}{--~}{
 \topsep=0.3ex \itemsep=0.3ex \labelsep=0em \parsep=0em
 \listparindent=1em \itemindent=0em
 \settowidth{\labelwidth}{--~} \leftmargin=\labelwidth
}}{\end{list}}

\newcommand{\R}{\mathbbm{R}}
\newcommand{\PP}{\mathcal{P}}

\newcommand{\kk}{\kappa}
\newcommand{\lm}{\lambda}
\newcommand{\ox}{\bar x}
\newcommand{\oy}{\bar y}
\newcommand{\ou}{\bar u}

\newcommand{\sign}{{\rm sign}\,}

\newcommand{\dsty}{\displaystyle}

\newtheorem{Theorem}{Theorem}[section]
\newtheorem{Remark}[Theorem]{Remark}

\title{On the Q-linear convergence of forward-backward splitting method and uniqueness of optimal solution to Lasso}

\author{J.Y. Bello-Cruz\thanks{Department of Mathematical Sciences, Northern Illinois University, DeKalb, IL 60115, USA. E-mail: yunierbello@niu.edu}\and G. Li \thanks{
Department of Applied Mathematics, University of New South Wales, Sydney
2052, Australia. E-mail: g.li@unsw.edu.au}  \and T.T.A.
Nghia\thanks{Department of Mathematics and Statistics, Oakland
University, Rochester, MI 48309, USA. E-mail: nttran@oakland.edu}}

\begin{document}
\maketitle
\begin{abstract}

\noindent In this paper, by using tools of second-order variational analysis,
we study the popular forward-backward splitting method with Beck-Teboulle's line-search for solving convex optimization problem where the objective function can be split into the sum of a differentiable function and a possible nonsmooth function. 
We first establish that this method exhibits global convergence to an optimal solution of the problem (if it exists) without the usual assumption that the gradient of the differentiable function involved is globally Lipschitz continuous.
We also obtain the  $o(k^{-1})$ complexity  for the functional value sequence when  this usual assumption is weaken from global Lipschitz continuity to local
Lipschitz continuity; improving the existing $\mathcal{O}(k^{-1})$ complexity result.
We then derive the local and global Q-linear convergence of the method in terms of both the function value sequence and the iterative sequence, under a general metric subregularity assumption which is automatically satisfied for convex piecewise-linear-quadratic optimization problems.  In particular, we provide verifiable sufficient conditions for
 metric subregularity
assumptions, and so, local and global Q-linear convergence of the proposed method for broad structured optimization problems arise in machine learning and signal processing including Poisson linear inverse problem, the partly smooth optimization problems,
as well as the $\ell_1$-regularized optimization problems. Our results complement the current literature by providing $Q$-linear convergence result to the forward-backward splitting method under weaker
assumptions. Moreover, via this approach, we obtain several full characterizations for the uniqueness of optimal solution to Lasso problem, which covers some recent results in this direction.
\end{abstract}


\vspace{-0.1in}
\section{ Introduction}
\setcounter{equation}{0}
In this paper we  consider the following optimization problem
\begin{equation}\label{prob}
\min_{x\in \R^n}\quad F(x):=f(x)+g(x),
\end{equation}
where $f,g:\R^n\to \R\cup\{\infty\}$ are proper, lower semi-continuous, and convex functions and $f$ is differentiable in its domain. Problems in this format have been appeared in many different fields of science and engineering including machine learning, compressed sensing, and image processing.  A particular class of \eqref{prob} known as $\ell_1$-regularized problem
\begin{equation}\label{Prof}
\min_{x\in \R^n}\quad F_1(x):=f(x)+\mu\|x\|_1
\end{equation}
with constant $\mu>0$ has been attracted huge attention and widely used in signal processing and statistics to derive sparse optimal solutions. One of the most popular cases of \eqref{Prof} is the  Lasso problem \cite{T} (also known as $\ell_1$-regularized least square optimization problem) formulated by
\begin{equation}\label{Pros}
\min_{x\in \R^n}\quad F_2(x):=\frac{1}{2}\|Ax-b\|^2+\mu\|x\|_1,
\end{equation}
where $A$ is an $m\times n$ matrix and $b$ is a vector in $\R^m$.

Among many methods of solving \eqref{prob}, the forward-backward splitting method (FBS in brief) \cite{BC,BL, BTe, CP, CW, HYZ, NB} is well-known due to its simplicity and efficiency as described below:
\begin{equation}\label{FBS}
x^{k+1}=\prox_{\al_k g}(x^k-\al_k\nabla f(x^k))
\end{equation}
with  the proximal operator defined later in \eqref{prox} and the stepsize $\al_k>0$. The global convergence of FBS to an optimal solution of problem \eqref{prob} and the complexity $\mathcal{O}(k^{-1})$ of the functional $F(x^k)$ to the minimum value are usually proved under the  assumption that $\nabla f$ is global Lipschitz continuous. By using some line searches motivated by the work of Tseng \cite{tseng}, \cite{YN} Bello-Cruz and Nghia show that  FBS indeed converges globally without the aforementioned Lipschitz condition, while  the complexity of functional value is improved to $o(k^{-1})$ when the $\nabla f$ is only {\em locally} Lipschitz continuous. A recent work of Bauschke-Bolte-Teboulle \cite{BBT} also tackles the absence of Lipschitz continuous gradient on $f$ by introducing the so-called NoLips algorithm close to FBS with the involvement of Bregman distance. Their algorithm also shares the sublinear complexity $\mathcal{O}(\frac{1}{k})$ of the functional sequence $(F(x^k))_{k\in\NN}$ under some mild assumptions, and guarantees the global convergence  of the solution sequence $(x^k)_{k\in\NN}$ with an additional hypothesis on the closedness of the domain of the auxiliary Legendre function defined there. Unfortunately, the latter assumption is not satisfied for the Poisson inverse regularized problems  with Kullback-Liebler divergence \cite{C,VSK}, one of the main applications in \cite{BBT}. This situation is overcome in our paper by revisiting FBS with the line search of Beck-Teboulle \cite{BTe}. Under some minimal assumptions on initial data weaker than those in \cite{YN,S17}, we show that the sequence $(x^k)_{k\in\NN}$ in \eqref{FBS} is globally convergent to an optimal solution (if it exists) without any Lipschitz continuity on the gradient of $f$. Moreover, the sublinear rate $o(\frac{1}{k})$ is obtained when the gradient $\nabla f$ is {\em locally} Lipschitz continuous on its domain, which is automatic in the case of Kullback-Liebler divergence.

 Our paper mainly devotes to the linear convergence of FBS.  Despite of the popularity of FBS, the linear convergence of this method has been established  recently throughout some  error bound conditions \cite{DL,NNG,ZS} with the base from \cite{LT} or Kurdya-\L ojasiewicz inequality \cite{BNPS,LP}. It is worth mentioning that those conditions are  somehow equivalent; see, e.g., \cite{BNPS,DL}. Our approach is close to the recent work of Drusvyatskiy-Lewis \cite{DL}, Bauschke-Phan-Noll \cite{BPN}, and Zhou-So \cite{ZS} by using the so-called {\em second-order growth condition}  and {\em metric subregularity} of the subdifferentials \cite{AC,AG,DMN}; however, our proof of linear convergence is more direct without using the error bound \cite{LT} and reveals the $Q$-linear convergence rather than the $R$-one obtained in all the aforementioned works.

Local linear convergence of FBS iterative sequence to solve some structured optimization problems of \eqref{prob} has been recently established in \cite{BL,LFP,LFP2,HYZ} when the function $g$ is {\em partly smooth relative to a manifold} $ \mathcal{M}$ by using the idea of finite support identification. This notion  introduced by Lewis \cite{L} allows Liang-Fadili-Peyr\'e \cite{LFP,LFP2} to cover in their work many important problems such as the total variation semi-norm, the $\ell_1$-norm \eqref{Prof}, the $\ell_\infty$-norm, and the nuclear norm problems. In their paper, a second-order condition was introduced to guarantee the Q-local linear convergence of FBS sequence generated by \eqref{FBS} under the non-degeneracy assumption \cite{L}. When applying our results to this structured setting, we only need a weaker condition. Using the calculus in \cite{LZ} is extremely helpful in computing the {\em second-order limiting subdifferential} \cite{M1} of a partly smooth function, but it technically sticks with the non-degeneracy assumption. When considering the $\ell_1$-regularized problem \eqref{Prof}, we are able to avoid this assumption and introduce a new second-order condition by employing  the recent result of Artacho-Geoffroy \cite{AG2} who initiate a new characterization for the  strong metrical subregularity of the subdifferential in term of graphical derivative \cite{DR}. This allows us to improve the  well-known work of Hale-Yin-Zhang \cite{HYZ} in two aspects: (a) We completely ignore the aforementioned non-degeneracy assumption (b) Our second-order condition is strictly weaker than the one in \cite[Theorem~4.10]{HYZ}. Our wider view is that when considering particular optimization problems listed in the spirit of \cite{LFP,LFP2},  the assumption of non-degeneracy may be not necessary. Furthermore, we revisit the {\em iterative shrinkage thresholding algorithm} (ISTA) \cite{BTe,dau}, which is indeed FBS for solving Lasso \eqref{Pros}. It is well-known that the complexity of this algorithm is $\mathcal{O}(k^{-1})$; however, the recent works \cite{LFP,TBZ} shows the potential of local linear convergence.  The stronger conclusion in this direction is obtained lately by   Bolte-Nguyen-Peypouquet-Suter \cite{BNPS} that: ISTA is   $R$-linearly convergent, but  the rate may depend on the initial point. Inspired by this achievement, we provide two new information: (c) Both functional sequence $(F(x^k))_{k\in \NN}$ and iterative sequence $(x^k)_{k\in \NN}$ from  ISTA are indeed   {\em globally} $Q$-linearly convergent  (d) They are eventually $Q$-linearly convergent to an optimal solution with a uniform rate that does not depend on the initial point. Another application of our work is solving Poisson inverse regularized problem \cite{BBT,C,VSK} by using FBS. We show the linear convergence of this method in contrast to the sublinear complexity $O(\frac{1}{k})$ obtained recently in \cite{BBT,S17} by different methods.

Finally, we study the uniqueness of optimal solution to Lasso problem as one of the main applications from our approach of using second-order variational analysis. This property of optimal solution to  \eqref{Pros} has been investigated vastly in the literature with immediate  implementations to recovering sparse signals in compressed sensing; see, e.g., \cite{F,Ti,Tr,Wa, ZYC, ZYY} and the references therein. It is also used in \cite{BL,TBZ} to establish the linear convergence of  ISTA. It seems to us that Fuchs \cite{F} initializes this direction by introducing a simple sufficient condition for this property, which has been extended in other cited papers. Then Tibshirani in \cite{Ti} shows that a  sufficient condition closely related to Fuchs' is also necessary for almost all $b$ in \eqref{Pros}.  The first full characterization  for this property has been obtained recently in \cite{ZYC} by using results of strong duality in linear programming. This characterization, which is based on an existence of a vector  satisfying a system of linear equations and inequalities, allows \cite{ZYC} to  recover the aforementioned sufficient conditions  and provide some situations in which these conditions turn necessary. As a direct application of our different approach, we also derive several new full characterizations. Our conditions in terms of {\em positively linear independence} and {\em Slater type} are well-recognized to be verifiable.

The outline of our paper is as follows. Section 2 briefly presents the relationship between the metric subregularity of the subdifferential, quadratic growth condition, and Kurdya-\L ojasiewicz inequality. A second-order characterization for quadratic growth condition in term of {\em graphical derivative} is also recalled here. This section serves as the main tool for us to obtain the linear convergence of FBS. The reader could find further details about this topic in \cite{AC,AG,AG2,DMN}. In Section 3, we provide the global convergence of FBS without the global Lipschitz condition on the gradient of $f$ and also the general complexity of $o(k^{-1})$. The study in this section is somewhat similar to the recent work \cite{YN}. However, we consider a different line search from those in \cite{YN} and our standing assumption is much weaker, which allows us to cover broader classes, e.g., the Poison inverse regularized problems studied in Section~5.1. The central part of our paper is Section 4, in which we show the Q-linear convergence of FBS under the metric subregularity of the subdifferential.  Section 5 devotes to many applications of our work to structured optimization problems involving Poison inverse regularized, partial smoothness, $\ell_1$-regularized, and
$\ell_1$-regularized least square optimization problems. In Section 6, we obtain several new full characterizations to the uniqueness of optimal solution to Lasso problem \eqref{Pros}. The final Section 7 gives the conclusions and some potential future works in this direction.

\section{Metric subregularity of the subdifferential and quadratic growth condition}\label{prelim}
Throughout the paper, $\R^n$ is the usual Euclidean space with dimension $n$ where $\|\cdot\|$ and $\langle \cdot, \cdot \rangle$ denote the corresponding Euclidean norm and inner product in $\R^n$. We use $\Gamma_0(\R^n)$ to denote the set of proper, lower semicontinuous, and convex functions on $\R^n$. Let $h\in \Gamma_0(\R^n)$, we write  $\dom h:=\{x\in\R^n\,|\; h(x)<+\infty\}$. The subdifferential of $h$ at $\ox\in \dom h$ is defined by
\begin{equation}\label{sub-inq}
\partial h(\ox):=\{v\in \R^n\,|\; \la v,x-\ox\ra\le h(x)-h(\ox),\; x\in \R^n\}.
\end{equation}
Let $G:\R^n\tto \R^m$ be a set-valued mapping. We define the domain and the graph of $G$, respectively as following:
\begin{equation*}
\dom G:=\{x\in \R^n|\; G(x)\neq\emptyset\}\quad \mbox{and}\quad \gph G:=\{(x,y)\in \R^n\times\R^m|\; y\in G(x)\}.
\end{equation*}

One of the key notions used in our paper is the so-called metric subregularity defined as follows; see \cite[Section~3H and 3I]{DR}.

\begin{definition}[Metric subregularity and strong metric subregularity]\label{MSR} We say $G:\R^n\tto\R^m$ is {\rm metrically subregular} at $\ox\in \dom G$ for $\oy\in G(\ox)$ with modulus $\kk>0$ if there exists a neighborhood $U\subset \R^n$ of $\ox$ such that
\begin{equation}\label{SR}
d(x;G^{-1}(\oy))\le \kk d(\oy; G(x))\quad \mbox{and}\quad x\in U,
\end{equation}
where $d(x;\Omega)$ is the distance from $x\in \R^n$ to a set $\Omega\subset\R^n$ with the convention that $d(x;\emptyset)=\infty$.
Furthermore, we say $G$ is {\em strongly metrically subregular} at $\ox$ for $\oy\in G(\ox)$ with modulus $\kk>0$  if $G$ is metrically subregular at $\ox$ for $\oy$ with modulus $\kk$ and $\ox$ is an isolated point of $G^{-1}(\oy)$.
\end{definition}

Metric subregularity is automatic when $G$ is a {\em piecewise polyhedral}  mapping, i.e., $\gph G$  the union of finitely
many convex polyhedral sets; see, e.g., \cite[Proposition~3H.1]{DR} and \cite[Example~9.57]{rw}, where its roots comes from Robinson \cite{R} and the landmark paper of Hoffman \cite{H}.
\begin{proposition}[Metric subregularity of polyhedral mappings] \label{MSP} Let $G:\R^n\tto\R^m$ be a piecewise polyhedral mapping. Then $G$ is metrically subregular at any $\ox\in \dom G$ for any $\oy\in G(\ox)$ with a uniform modulus $\kk>0$ that does not depend on $(\ox,\oy)$.
\end{proposition}
\begin{proof}
Since $G:\R^n\tto\R^m$ is  piecewise polyhedral, its inverse is
also piecewise polyhedral. Combining this with \cite[Proposition~3H.1 and Proposition~3H.3]{DR} tells us that $G$ is metrically subregular at any $\ox\in \dom G$ for any $\oy\in G(\ox)$ with a uniform modulus $\kk>0$.
\end{proof}

From the above result, $G$ is strongly  metric subregular at $\ox$ for $\oy\in G(\ox)$ when $G$ is piecewise polyhedral and $\ox$ is an isolated point of $G^{-1}(\oy)$.
Without polyhedrality, metric subregularity may be difficult to check; however, strong metric subregularity can be characterized and verified directly via the so-called {\em graphical derivative} (known also as {\em contingent derivative}); see \cite[Section~4A]{DR}.
\begin{definition}[Graphical derivative]\label{GD} Let $(\ox,\oy)\in \gph G$. The {\em graphical derivative} of $G$ at $\ox$ for $\oy$ is the mapping $DG(\ox|\oy):\R^n\tto\R^m$ such that $v\in DG(\ox|\oy)(u)$ if and only if there exist sequences $u^k\to u$, $v^k\to v$ and $t^k\dn 0$ such that $\oy+t^kv^k\in G(\ox+t^ku^k)$ for all $k\in \NN$.
\end{definition}

When ${\rm gph}\, G$ is locally closed around $(\ox,\oy)\in \gph G$, it is known from \cite[Theorem~4C.1]{DR} that $G$ is strongly metrically subregular at $\ox$ for $\oy$ if and only if
\begin{equation}\label{Levy}
DG(\ox|\oy)^{-1}(0)=\{0\}.
\end{equation}

Another important stability notion useful in our study is the so-called {\em metric regularity}  \cite[Sections~3G]{DR}  or  \cite[Definition~1.47]{M1}.

\begin{definition}[Metric regularity]\label{MTR} The set-valued mapping $G:\R^n\tto\R^m$ is {\rm metrically regular} at $\ox$ for $\oy\in G(\ox)$ with modulus $\kk>0$ if there exist neighborhoods $U\subset \R^n$ of $\ox$  and $V\subset \R^m$ of $\oy$ such that
\begin{equation}\label{MR}
d(x;G^{-1}(y))\le \kk d(y; G(x))\quad \mbox{for all}\quad (x,y)\in U\times V.
\end{equation}
Furthermore, we say $G$ is {\em strongly metrically regular} at $\ox$ for $\oy$ with modulus $\kk>0$  if there exist neighborhoods $U\subset \R^n$ of $\ox$  and $V\subset \R^m$ of $\oy$ such that \eqref{MR} is satisfied and that the map $V\ni y\mapsto G^{-1}(y)\cap U$ is single-valued.
\end{definition}

Metric regularity and strong metric regularity could be characterized fully in \cite[Theorem~4.18]{M1} and \cite[Theorem~9.40]{rw}.  However, we do not use these infinitesimal characterizations in the paper.

In later sections we mainly employ the metric subregular property of the subdifferential mapping $\partial h:\R^n\tto \R^n$ for $h\in \Gamma_0(\R^n)$ to derive the linear convergence of FBS method. We need the strict connection between metric subregularity of $\partial h$ and the growth condition established recently in \cite{AC,AG, AG2,DMN}.

\begin{proposition}[Metric subregularity of the subdifferential and growth condition]\label{SR-grow} Let $h\in \Gamma_0(\R^n)$ and $\ox$ be an optimal solution to $h$, i.e., $0\in \partial h(\ox)$.  Consider the following assertions:
\begin{enumerate}[{\rm (i)}]

\item $\partial h$ is metrically subregular at $\ox$ for $0$ with modulus $\kk>0$.

\item  $h$ satisfies the second-order growth condition at $\ox$ in the sense that: there exist   $c,\ve>0$ such that
\begin{equation}\label{grow1}
h(x)\ge h(\ox)+\frac{c}{2}d^2(x;(\partial h)^{-1}(0))\quad \mbox{for all}\quad x\in \B_\ve(\ox).
\end{equation}
\end{enumerate}
Then implication {\rm[(i)$\Longrightarrow$ (ii)]} holds, where $c$ can be chosen as $\kk^{-1}$ in \eqref{grow1}. Moreover, the converse implication {\rm [(ii)$\Longrightarrow$(i)]} is also satisfied with $\kk=\frac{2}{c}$.
\end{proposition}
\begin{proof}
Applying \cite[Theorem~6.2(a)]{AC} for  convex functions clarifies the implication {\rm[(i)$\Longrightarrow$ (ii)]}. The converse implication follows from either \cite[Theorem~3.3]{AG2} or \cite[Theorem~6.2(b)]{AC}.
\end{proof}

The equivalence between {\rm (i)} and {\rm (ii)} above was first established in \cite[Theorem~3.3]{AG} in Hilbert spaces without studying the connection of $\kk$ and $c$. The result has been improved and extended to Asplund spaces in \cite{DMN} even for the case of nonconvex functions with further investigations on the modulus. The recent work \cite[Theorem~6.2]{AC} has derived
the tightest relationship between $\kk$ and $c$ as described in Proposition~\ref{SR-grow}. The second-order growth condition \eqref{grow1} is also called {\em quadratic functional growth} property in  \cite{NNG} when $h$ is continuously differentiable over a closed convex set.

A similar statement to Proposition~\ref{SR-grow} was  also used recently in \cite[Theorem~3.3]{DL} to prove the linear convergence of the classical forward-backward splitting method. Their quadratic growth condition  is slightly different as follows:
\begin{equation}\label{growDL}
h(x)\ge h(\ox)+\frac{c}{2}d^2(x;(\partial h)^{-1}(0))\quad \mbox{for all}\quad x\in [h < h^*+\nu]
\end{equation}
for some constants $c, \nu>0$, where $[h< h^*+\nu]=\{x\in \R^n|\; h(x)< h^*+\nu\}$ and $h^*:=\inf h(x)=h(\ox)$. It is easy to check that this growth condition implies \eqref{grow1}. Indeed, suppose that \eqref{growDL} is satisfied for some $c,\nu>0$ Define $\eta:=\sqrt{\frac{2\nu}{c}}$ and note that $\ox\in [h< h^*+\nu]$. Take any $x\in \B_\eta(\ox)$, if $x\in [h< h^*+\nu]$, inequality \eqref{grow1} is trivial. If    $x\notin [h< h^*+\nu]$, it follows that
\[
h(x)\ge h(\ox)+\nu= h(\ox)+\frac{c}{2}\eta^2\ge h(\ox)+\frac{c}{2}\|x-\ox\|^2\ge  h(\ox)+\frac{c}{2}d^2(x;(\partial h)^{-1}(0)),
\]
which  clearly verifies \eqref{grow1}. Thus \eqref{grow1} is weaker than \eqref{growDL}, but it is equivalent to the local version of \eqref{growDL} described as below:
\begin{equation}
h(x)\ge h(\ox)+\frac{c}{2}d^2(x;(\partial h)^{-1}(0))\quad \mbox{for all}\quad x\in [h < h^*+\nu] \cap\B_\ve(\ox)\nonumber
\end{equation}
for some constants $c, \nu,\ve>0$. This property has been showed recently in \cite[Theorem~5]{BNPS} to be equivalent to the fact that the convex function  $h$ satisfies the Kurdyka-\L ojasiewicz inequality with order $\frac{1}{2}$  recalled below.

\begin{definition}[Kurdyka-\L ojasiewicz inequality with order $\frac{1}{2}$]\label{KL1}
Let $h\in \Gamma_0(\R^n)$ and $\ox$ be an optimal solution to $h$, i.e., $0\in \partial h(\ox)$.  We say $h$ satisfies the
{\em Kurdyka-\L ojasiewicz inequality} with order $\frac{1}{2}$ at $\ox$ if there exist some $\nu,\ve,c>0$ such that
\begin{equation}\label{KL}
d(0;\partial h(x))\ge c [h(x)-h(\ox)]^\frac{1}{2}\quad \mbox{for all}\quad  x\in [h< h^*+\nu]\cap\B_\ve(\ox).
\end{equation}
\end{definition}

From the above discussion,  $h$ satisfies the {\em Kurdyka-\L ojasiewicz inequality} with order $\frac{1}{2}$ at $\ox$ if and only $\partial h$ is metrically subregular at $\ox$ for $0$. However, throughout the paper, we mainly use the metric subregularity of $\partial h$ or the quadratic growth condition to reveal some new information for FBS \eqref{FBS} and the uniqueness of optimal solution to Lasso problem \eqref{Pros}.

There are also similar characterizations for the strong metric subregularity of $\partial h$ as discussed below, in which the positive-definiteness  of $D\partial h(\ox|0)$ is introduced in \cite{AG2}.

\begin{proposition}[Strong metric subregularity of the subdifferential and growth condition]\label{SR-grow2} Let $h\in \Gamma_0(\R^n)$ and $\ox$ be an optimal solution, i.e., $0\in \partial h(\ox)$.  Consider the following assertions:
\begin{enumerate}[{\rm(i)}]
\item $\partial h$ is strongly metrically subregular at $\ox$ for $0$ with modulus $\kk>0$.

\item There exist $c,\eta>0$ such that
\begin{equation}\label{grow2}
h(x)\ge h(\ox)+\frac{c}{2}\|x-\ox\|^2\quad \mbox{for all}\quad x\in \B_\eta(\ox).
\end{equation}

\item $D\partial h(\ox|0)$ is positive-definite with the modulus $\ell>0$ in the sense that
\begin{equation}\label{pos}
\la v,u\ra\ge \ell \|u\|^2\quad \mbox{for all}\quad  v\in D(\partial h)(\ox|0)(u), u\in \R^n.
\end{equation}
\end{enumerate}
Then the implication {\rm [(i)$\Longrightarrow$ (ii)]} holds with $c$ in \eqref{grow2} chosen as $\kk^{-1}$. If {\rm (ii)} is  satisfied, {\rm(iii)} is also valid with $\ell=\frac{c}{2}$. We also have the implication {\rm [(iii)$\Longrightarrow$ (i)]} with $\kk>\ell^{-1}$. Moreover, {\rm (iii)} holds if and only if
\begin{equation}\label{pos4}
\la v,u\ra>0\quad \mbox{for all}\quad  v\in D(\partial h)(\ox|0)(u), u\in \R^n, u\neq 0.
\end{equation}

Consequently, if one of {\rm (i)}, {\rm (ii)}, and {\rm (iii)} is fulfill then $\ox$ is a unique optimal solution to $h$.
\end{proposition}
 \begin{proof}
 The equivalences of  {\rm (i)}, {\rm (ii)}, and {\rm (iii)} follow from \cite[Theorem~3.6 and Corollary~3.7]{AG2}. Moreover, if {\rm (i)} is satisfied, then $c$ in \eqref{grow2} can be chosen as $\kk^{-1}$ by Proposition~\ref{SR-grow} and Definition~\ref{MSR}; see also \cite[Corollary~6.1]{AC}. When {\rm (iii)} holds, it follows from \cite[Corollary~3.7]{AG2} that {(i)} is satisfied for any $\kk>\ell^{-1}$. To complete the first part of the proposition, we only need to verify that the validity of  \eqref{pos4} implies the existence of $\kk$ in \eqref{pos}.  Indeed, it follows from \eqref{pos4} that $D(\partial f)(\ox,0)^{-1}(0)=\{0\}$. Thanks to   \eqref{Levy}, $\partial h$ is strongly metrically subregular at $\ox$ for $0$, which means {(i)} and obviously implies {\rm (iii)}.

 Now suppose that one of {\rm (i)}, {\rm (ii)}, and {\rm (iii)} is fulfilled. Thus $\ox$ is an isolated point of $\partial h^{-1}(0)$ by Definition~\ref{MSR}. Note also that $(\partial h)^{-1}(0)$ is a convex set in $\R^n$. It follows that $(\partial h)^{-1}(0)=\{\ox\}$. The proof is complete.
 \end{proof}

It is clear from the discussion after Definition~\ref{GD}, \eqref{Levy}, and the above result that  $D(\partial h)(\ox|0)^{-1}(0)=\{0\}$ is also equivalent  to  \eqref{pos}. Nevertheless, using \eqref{pos} is somehow more convenient in some applications discussed later in Section~5.

Next let us discuss the metric regularity of the subdifferential with the connection with tilt stability introduced by Poliquin-Rockafellar \cite{PR2}.
\begin{proposition}[Metric regularity of the subdifferential and tilt stability]\label{tilt} Let $h\in \Gamma_0(\R^n)$ and $\ox$ be an optimal solution.  The followings are equivalent:
\begin{enumerate}[{\rm(i)}]
\item $\partial h$ is metrically regular at $\ox$ for $0$ with the modulus $\kk>0$.

\item $\partial h$ is strongly metrically regular at $\ox$ for $0$ with the modulus $\kk>0$.

\item $\ox$ is a {\em tilt stable} local minimizer with modulus $\kk>0$  to $h$ in the sense that  there exists $\gm>0$ such that the mapping
\[
M_\gm:v\mapsto {\rm argmin}\, \left\{ f(x)-\la v,x\ra|\; x\in \B_\gm(\ox)\right\}
\]
is single-valued and Lipschitz continuous with constant $\kk$ on some neighborhood of $0$ with $M_\gm(0)=\{\ox\}$.
\end{enumerate}
\end{proposition}
\begin{proof}
The equivalence of {\rm (i)} and {\rm (ii)} follows from  \cite[Proposition~3.8]{AG} or \cite[Theorem~3G.5]{DR}. Moreover, the equivalence between {\rm (ii)} and {\rm (iii)} is derived from \cite[Proposition~4.1]{MN}.
\end{proof}

To complete this section, we recall a few important notions of linear convergence in our study. A sequence $(x^k)_{k\in \NN}\subset \R^n$ is called to be R-linearly convergent to $x^*$ with rate $\mu\in (0,1)$ if there exists $M>0$ such that
\[
\|x^k-x^*\|\le M\mu^k=\mathcal{O}(\mu^k) \quad\mbox{for all}\quad k\in \NN.
\]
We say $(x^k)_{k\in \NN}$ is {\em Q-linearly convergent} to $x^*$ with rate $\mu\in (0,1)$ if there exists $K\in \NN$ with
\[
\|x^{k+1}-x^*\|\le \mu \|x^k-x^*\|\quad \mbox{for all}\quad k\ge K.
\]
Furthermore, $(x^k)_{k\in \NN}$ is {\em globally} Q-linearly convergent to $x^*$ with rate $\mu\in (0,1)$ if
\[
\|x^{k+1}-x^*\|\le \mu \|x^k-x^*\|\quad \mbox{for all}\quad k\in \NN.
\]
It is obvious that the global Q-linear convergence implies the Q-linear convergence. Moreover, R-linear convergence holds under the validity of local Q-linear convergence with the same rate. On the other hand, R-linear convergence does not imply Q-linear convergence. A simple example is when
$a_k:=\|x^k-x^*\|$ satisfies
\[
a_k=\left\{\begin{array}{ccl}
\epsilon^k & \mbox{ if } & k \mbox{ is odd},\\
\epsilon^{2k} & \mbox{ if } & k \mbox{ is even},
\end{array}
\right.
\]
where $\epsilon \in (0,1)$. It is clear that $(x^k)_{k \in \NN}$ is R-linearly convergent with rate $\epsilon$; while it is not Q-linearly convergent.

\section{Global convergence of forward-backward splitting methods}
In this section, we recall the theory for forward-backward splitting methods (FBS) when the gradient $\nabla f$ is not globally Lipschitz continuous. The results here are somewhat similar to \cite{YN,S17} with slight relaxations on the standing assumptions.  This section provides some facts used in our Section 4 and 5 to establish the Q-linear convergence of FBS.

Let us start with  the standing assumptions on the initial data for   problem \eqref{prob} used throughout the paper:
\begin{enumerate}
\item[{\bf A1.}] \label{a1} $f, g\in \Gamma_0(\R^n)$\mbox{ and } ${\rm int}(\dom f)\cap  \dom g\neq \emptyset$.
\item[{\bf A2.}] \label{a2} For any $x\in {\rm int}(\dom f)\cap  \dom g$, the sublevel set $\{F\le F(x)\}$ is contained in ${\rm int}(\dom f)\cap  \dom g $ and $f$ is continuously differentiable at any point in  $\{F\le F(x)\}$ with $F(\cdot)=(f+g)(\cdot)$.
\end{enumerate}

Our assumptions are certainly less restrictive than the standard ones broadly used in the theory of FBS \eqref{FBS} \cite{BTe,CP,CW}:
\begin{enumerate}
\item [{\bf H1.}]  $f, g\in \Gamma_0(\R^n)$ and $\dom g\subset \dom f$.
\item [{\bf H2.}]  $f\in \mathcal{C}^{1,1}$, i.e.,  $\nabla f$ is globally Lipschitz continuous.
\end{enumerate}
It is worth noting further that both assumptions {\bf A1} and {\bf A2} are valid whenever $f$ is continuously differentiable on  ${\rm int}(\dom f)\cap  \dom g$ and ${\rm int}(\dom f)\cap  \dom g=(\dom f)\cap  \dom g\neq \emptyset$. In the Poisson inverse regularized problem discussed in Section~5.1, the latter condition is trivial, while both {\bf H1} and {\bf H2} are not satisfied. Our conditions are also strict relaxations of the following ones proposed recently in \cite[Section~4]{S17}:
\begin{enumerate}
\item[{\bf H1$^\prime$.}]  $f, g\in \Gamma_0(\R^n)$
are bounded from below with
$ \mathrm{int}\, (\dom f) \cap \dom g\neq \emptyset$.
\item[{\bf H2$^\prime$.}] $f$ is differentiable on
$ \mathrm{int}(\dom f)\cap \dom g$, $\nabla f$
is uniformly continuous on any compact subset of
$ \mathrm{int}(\dom f)\cap \dom g$, and $\nabla f$ is bounded on any
sublevel sets of $F$.
\item[{\bf H3$^\prime$.}] 
For every $x \in {\rm int}(\dom f)\cap \dom g$,
$\{F \leq F(x)\} \subset  {\rm int}(\dom f)\cap \dom g $ and
 $d(\{F \leq F(x)\};\R^n\setminus {\rm int}(\dom f))>0$.
\end{enumerate}
It is easy to see that some  unnatural boundedness in {\bf H1$^\prime$} and {\bf H2$^\prime$}  and the positive distance gap requirements in {\bf H3$^\prime$} are completely removed in our study. Moreover, our assumptions ({\bf A1}--{\bf A2}) are indeed strictly weaker than ({\bf H1$^\prime$}--{\bf H3$^\prime$}). To see this, consider  $p\in\{1,2\}$,
$$C_p:=\{x=(x_1,x_2) \in \R^2\,|\; x_1x_2 \ge p, x_1 \ge 0, x_2 \ge 0\}, $$
 and $f,g:\R^2 \rightarrow \R \cup\{+\infty\}$ be defined as $f(x)=-\log x_1-\log x_2$ if $x \in \R^2_{++}$ and $+\infty$ otherwise, and $g(x)=\delta_{C_1}(x)$ where $\delta_{C_1}$ is the indicator function of the set $C_1$ defined above. It can be directly verified that the assumptions ({\bf A1}--{\bf A2}) are satisfied. On the other hand, note that  {\bf H1$^\prime$} fails because $f$ is unbounded on
 $${\rm int} (\dom\, f) \cap \dom g=\R^2_{++}\cap C_{1}=C_1=\{(x_1,x_2) \in \R^2\,|\; x_1x_2 \ge 1, x_1 \ge 0, x_2 \ge 0\}.$$ Moreover, $(2,1) \in {\rm int} (\dom\, f) \cap \dom g$ and since $F(2,1)=(f+g)(2,1)=-\log 2$, we get  $$\{x\,|\; F(x) \le -\log 2\}=\{(x_1,x_2) \in \R^2\,|\; x_1x_2 \ge 2, x_1 \ge 0, x_2 \ge 0\}=C_2.$$So, clearly {\bf H2$^\prime$} also fails  because $\nabla f(x_1,x_2)=(-\frac{1}{x_1},-\frac{1}{x_2})$ is unbounded on the above sublevel set of $F$. Finally, 
observing that $(\frac{2}{k},k) \in C_2$ and $(0,k) \in \R^2 \backslash {\rm int\, (dom}\, f)=\R^2\backslash \R^2_{++}$, we see that $$d\big(\{F \le -\log 2\}; \R^2 \backslash {\rm int\, (dom}\, f)\big)=d\big(C_2;\R^2\backslash \R^2_{++}\big)\le \lim_{k\to \infty}\left\|(\frac{2}{k},k)-(0,k)\right\|=\lim_{k\to \infty} \frac{2}{k}=0.$$ This shows that
 {\bf H3$^\prime$} fails in this case.

Next let us recall  the proximal operator $\prox_{g}:\R^n\to \dom g$ given by
\begin{equation}\label{prox}
\prox_g(z):=(\Id+\partial g)^{-1}(z)\quad \mbox{for all}\quad  z\in \R^n,
\end{equation}
which is  well-known to be a single-valued mapping with full domain. With $\al>0$, it is easy to check that
\begin{equation}\label{in-sub}
\frac{z-\prox_{\alpha g}(z)}{\al}\in \partial g(\prox_{\al g}(z))\quad\mbox{for all} \quad z\in\R^n.
\end{equation}
Let $S^*$ be the optimal solution set to problem \eqref{prob} and $x^*$ be an element in ${\rm int}(\dom f)\cap \dom g$. Then $x^*\in S^*$ if and only if
\begin{equation}\label{sumr}
0\in \partial (f+g)(x^*)=\nabla f(x^*)+\partial g(x^*).
\end{equation}
The following lemma is helpful in our proof of the finite termination of   Beck--Teboulle's line search under our standing assumptions.
\begin{lemma}
\label{lem:prox}
Let $g \in\Gamma_0(\R^n)$ and let $\alpha >0$.
Then, for every $ x \in \dom g$, we have
\begin{equation}
\label{suco}
\prox_{\alpha g}(x) \to x\quad \mbox{as}\quad \al\to 0^+.
\end{equation}
\end{lemma}
\begin{proof}
Let $x \in  \dom g$ and $\alpha> 0$. Define  $z=\prox_{\alpha g}(x)$, we derive from \eqref{in-sub} that
$\frac{x-z}{\al}\in \partial g(z),$ which implies that
\begin{equation}\label{susu1}
g(x)-g(z)\ge \left\la \frac{x-z}{\al},x-z\right\ra=\frac{1}{\al}\|x-z\|^2.
\end{equation}
Since $g$ is proper, l.s.c. and convex, we have
\[
\infty>g(x)=g^{**}(x)=\sup_{u\in\R^n}\{\la u,x\ra-g^*(u)\},
\]
which is the Fenchel biconjugate of $g$ at $x$. Hence, for any $\ve>0$, there exists $u\in \R^n$ such that
\[
g(x)\le \la u,x\ra-g^*(u)+\ve\le \la u,x\ra-\la u,z\ra+g(z)+\ve.
\]
Combining this with \eqref{susu1} gives us that
\[
\|x-z\|^2\le \al(\la u, x-z\ra+\ve)\le \al\|u\|\cdot\|x-z\|+\al\ve,
\]
which implies that
\[
\|x-\prox_{\al g}(x)\|=\|x-z\|\le \dfrac{\al\|u\|+\sqrt{\al^2\|u\|^2+4\al\ve}}{2}.
\]
Since both $u$ and $\ve$ do not depend on $\al$, taking $\al\to 0^+$ from the latter inequality verifies \eqref{suco}.
\end{proof}

Under our standing assumptions, we define the {\em proximal forward-backward operator} $J:\big[{\rm int}(\dom
f)\cap \dom g\big]\times \R_{++}\to \dom g$   by
\begin{equation}\label{J}
J(x,\alpha):=\prox_{\alpha g}(x-\alpha \nabla f(x))\quad \mbox{for
all}\quad x\in {\rm int}(\dom f)\cap \dom g,\,\al>0.
\end{equation}
The following result  is essentially from \cite[Lemma~2.4]{YN}. Since the standing assumptions are different, we provide the proof for completeness.
\begin{lemma}\label{lema-prop} For any $x\in {\rm int} (\dom f)\cap \dom g$, we have
\begin{equation}\label{ineq-1}
\frac{\alpha_2}{\alpha_1}
\|x-J(x,\alpha_1)\|\ge\|x-J(x,\alpha_2)\|\ge \|x-J(x,\alpha_1)\|\quad \mbox{for all}\quad \alpha_2\ge \alpha_1>0.
\end{equation}
\end{lemma}
\begin{proof}
By using \eqref{in-sub} and \eqref{J} with $z=x-\alpha \nabla f(x)$, we have
\begin{equation}\label{inc-1}
\frac{x-\alpha \nabla f(x)-J(x,\alpha)}{\alpha}\in \partial g(J(x,\alpha))
\end{equation} for all $\alpha>0$. For any $\alpha_2\ge \alpha_1>0$, it follows from the monotonicity of $\partial g$ and \eqref{inc-1} that
\begin{align*}
0\le &\left \la \frac{x-\alpha_2 \nabla
f(x)-J(x,\alpha_2)}{\alpha_2}-\frac{x-\alpha_1 \nabla
f(x)-J(x,\alpha_1)}{\alpha_1},J(x,\alpha_2)-J(x,\alpha_1)\right\ra\\
=&\left\la
\frac{x-J(x,\alpha_2)}{\alpha_2}-\frac{x-J(x,\alpha_1)}{\alpha_1},\left(x-J(x,\alpha_1)\right)-\left(x-J(x,\alpha_2)\right)\right\ra\\
=&-\frac{\|x-J(x,\alpha_2)\|^2}{\alpha_2}-\frac{\|x-J(x,\alpha_1)\|^2}{\alpha_1}+\left(\frac{1}{\alpha_2}+\frac{1}{\alpha_1}\right)\la
x-J(x,\alpha_2),x-J(x,\alpha_1)\ra\\ \le&
-\frac{\|x-J(x,\alpha_2)\|^2}{\alpha_2}-\frac{\|x-J(x,\alpha_1)\|^2}{\alpha_1}+\left(\frac{1}{\alpha_2}+\frac{1}{\alpha_1}\right)\|x-J(x,\alpha_2)\|\cdot\|x-J(x,\alpha_1)\|,
\end{align*}
which easily imply the following expression
$$
\big(\|x-J(x,\alpha_2)\|-\|x-J(x,\alpha_1)\|\big)\cdot\Big(\|x-J(x,\alpha_2)\|-\frac{\al_2}{\al_1}\|x-J(x,\alpha_1)\|\Big)\le 0.
$$
Since $\dsty\frac{\alpha_2}{\alpha_1}\ge 1$, we derive
\eqref{ineq-1} and thus complete the proof of the lemma.
 \end{proof}

Next, let us present Beck-Teboulle's backtracking line search \cite{BTe}, which is specifically useful for forward-backward methods when the Lipschitz constant of $\nabla f$ is not known or hard to estimate.
\begin{center}\fbox{\begin{minipage}[b]{\textwidth}
\label{boundary}\noindent {\bf Linesearch BT} (Beck--Teboulle's line search)\\
Given $x\in {\rm int}(\dom f)\cap \dom g$, $\sigma>0$ and $\theta\in (0,1)$.\\
{\bf Input.} Set $\alpha=\sigma$ and $J(x,\alpha)=\prox_{\alpha
g}(x-\alpha \nabla f(x))$ with $x\in \dom g$.
\begin{retraitsimple}
\item[] {\bf While} $ \displaystyle
f(J(x,\alpha))> f(x)+\la\nabla
f(x),J(x,\alpha)-x\ra+\frac{1}{2\alpha}\|x-J(x,\alpha)\|^2,$   {\bf do}\\
$\alpha=\theta \alpha$.
\item[] {\bf End While}
\end{retraitsimple}
{\bf Output.} $\alpha$.
\end{minipage}}
\end{center}

The output $\al$ in this line search will be denoted by $LS(x,\sigma,\theta)$. Let us show the well-definedness and finite termination of this line search under the standing assumptions {\bf A1} and {\bf A2} below.

\begin{proposition}[Finite termination of Beck--Teboulle's line search] \label{boundary-well} Suppose that assumptions {\bf A1} and {\bf A2} hold. Then,  for any  $x\in  {\rm int}(\dom f)\cap \dom g$, we have
\begin{enumerate}[{\rm (i)}]
 \item The above line search  terminates after finitely many iterations with the positive output $\bar\alpha=LS(x,\sigma,\theta)$.

\item $\|x-u\|^2-\|J(x,\bar \al)-u\|^2\ge 2\bar \al[F(J(x,\bar \al))-F(u)]$ for any $u\in \R^n$.

 \item  $F(J(x,\bar \al))-F(x)\le -\dfrac{1}{2\bar \al}\|J(x,\bar \al)-x\|^2\le 0$. Consequently, $J(x,\bar \al)\in {\rm int} (\dom f)\cap \dom g$ and $f$ is continuously differential at $J(x,\bar \al)$.
 \end{enumerate}
\end{proposition}
\begin{proof} Take any $x\in{\rm int}(\dom f)\cap \dom g$.  Let us justify {\rm (i)} first. Note that $J(x,\alpha)$ is well-defined for any $\alpha>0$ because $f$ is differentiable at $x$ by assumption {\bf A2}.
  If $x\in S^*$, where $S^*$ is the optimal solution set to problem \eqref{prob}, then $x=J(x,\sigma)$ due to \eqref{sumr} and \eqref{in-sub}. Thus the line search stops with zero step and gives us the output $\sigma$ and  $x=J(x,\sigma)\in {\rm int}(\dom f)\cap \dom g$.
If $x\notin S^*$, suppose by contradiction that the line search does not terminate after finitely many steps.
Hence,  for all $\alpha\in\PP:=\{\sigma, \sigma\theta, \sigma\theta^2, \ldots \}$ it follows that
\begin{equation} \label{fisrt-Ineq}
\la\nabla
f(x),J(x,\alpha)-x\ra+\frac{1}{2\alpha}\|x-J(x,\alpha)\|^2<f(J(x,\alpha))- f(x).
\end{equation}
Since
 $\prox_{\alpha g}$ is non-expansive, we have
\begin{equation}
\label{eq:20160213a}
\begin{aligned}
\|J(x,\alpha) - x\|&\leq \|\prox_{\alpha g}(x - \alpha \nabla f(x))
- \prox_{\alpha g}(x)\| + \|\prox_{\alpha g} (x)- x\|\\
& \leq \alpha \|\nabla f(x)\| + \|\prox_{\alpha g} (x)- x\|.
\end{aligned}
\end{equation}
Due to the fact that  $\nabla f(x)$ and  $g(x)$ are finite,  Lemma~\ref{lem:prox} tells us that \begin{equation}\label{jxagoestox}\|J(x,\alpha) - x\| \to 0\quad
\mbox{as} \quad \alpha \to 0^+.\end{equation}
 Since $x\in {\rm int}(\dom f)$, there exists $\ell\in\NN$ such that $J(x,\alpha)\in  {\rm int}(\dom f)$ for all $\alpha\in\PP':=\{\sigma\theta^\ell, \sigma\theta^{\ell+1}, \ldots \}\subseteq \PP$.  Thus $J(x,\alpha)\in  {\rm int}(\dom f)\cap \dom g$ for all $\alpha\in\PP'$. Thanks to the convexity of $f$, we have
\begin{equation} \label{second-Ineq}
f(x)-f(J(x,\alpha))\ge \la\nabla f(J(x,\alpha)),x-J(x,\alpha)\ra \;\;\; \mbox{for}\; \;\;\al\in \PP'.
\end{equation}
This inequality together with \eqref{fisrt-Ineq} implies
\begin{align*}
\frac{1}{2\al}\left\|J(x,\alpha)-x\right\|^2&<\la\nabla f\big(J(x,\alpha)\big)-\nabla f(x),J(x,\alpha)-x\ra\\&\le\left\|\nabla f\big(J(x,\alpha)\big)-\nabla f(x)\right\|\cdot \left\|J(x,\alpha)-x\right\|,\end{align*}
 which yields  $J(x,\alpha) \neq x$ and
\begin{eqnarray}\label{Jal}
0<\frac{\left\|J(x,\alpha)-x\right\|}{\alpha}<2\left\|\nabla f\big(J(x,\alpha)\big)-\nabla f(x)\right\|\;\; \mbox{for all}\;\; \al\in \PP'.\;\;\;\;
\end{eqnarray}
Since $ \|x-J(x,\alpha)\|\to 0$ as
$\al\to  0$ by \eqref{jxagoestox} and $\nabla f$ is continuous on  ${\rm int}(\dom f)\cap \dom g$ by Assumption~{\bf A2}, we obtain from \eqref{Jal} that
\begin{equation}\label{eq-1.1}
\lim_{\alpha\to  0,\al\in \PP'}\frac{\|x-J(x,\alpha)\|}{\alpha}=0.
\end{equation}
Applying  \eqref{in-sub} with $z=x-\alpha \nabla f(x)$ gives us that
$\displaystyle
\frac{x-J(x,\alpha)}{\alpha}\in \nabla
f(x)+\partial g(J(x,\alpha)).
$
It follows from the convexity of $g$ that
\[
\left\la \frac{x-J(x,\alpha)}{\alpha}-\nabla f(x), y-J(x,\alpha)\right\ra\le g(y)-g(J(x,\al))\quad \mbox{for all}\quad y\in \R^n.
\]
Since the function $g$ is lower semicontinuous, after taking $\al\to 0$, $\al\in\mathcal{P}'$, we have
\begin{equation*}\label{xy}
\la -\nabla f(x), y-x\ra\le g(y)-g(x)\quad \mbox{for all}\quad y\in \R^n,
\end{equation*}
which yields $-\nabla f(x)\in \partial g(x)$, i.e., $0\in \nabla f(x)+ \partial g(x)$.
 This contradicts the hypothesis that $x\notin S^*$ by \eqref{sumr}. Hence, the line search terminates after finitely many steps with the output $\bar\alpha$ .

 To proceed the proof of {\bf (ii)}, note that
 \begin{equation}\label{fin}
 f(J(x,\bar\alpha))\le f(x)+\la\nabla
f(x),J(x,\bar\alpha)-x\ra+\frac{1}{2\bar\alpha}\|x-J(x,\bar\alpha)\|^2.
 \end{equation}
Moreover, by \eqref{in-sub} , we have
\begin{equation*}
\frac{x-J(x,\bar \al)}{\bar \alpha}-\nabla f(x)\in\partial g(J(x,\bar \al))=\partial g(J(x,\bar \al)).
\end{equation*}
Pick any $u\in \R^n$, we get from the later that
\begin{equation}\label{eq2}g(u)-g(J(x, \bar \al))\ge \left\la\frac{x-J(x, \bar \al)}{\bar \alpha}-\nabla f(x),u-J(x, \bar \al)\right\ra.
\end{equation}
Observe further that
\begin{equation}\label{eq3}
f(u)-f(x)\ge \la\nabla f(x), u-x\ra.
\end{equation}
Adding \eqref{eq2} and \eqref{eq3} and using \eqref{fin} give us that
\begin{align*}
 F(u)=(f+g)(u)\ge& f(x)+g(J(x, \bar \al)) + \left\la\frac{x-J(x, \bar \al)}{\bar \al}-\nabla f(x),u-J(x, \bar \al)\right\ra+\la\nabla f(x), u-x\ra\\
 =& f(x)+g(J(x, \bar \al)) +\frac{1}{\bar \al} \la x-J(x, \bar \al),u-J(x, \bar \al)\ra+\la\nabla f(x), J(x, \bar \al)-x\ra \\
 \ge&f(J(x, \bar \al))+g(J(x, \bar \al)) +\frac{1}{\bar\alpha} \la x-J(x, \bar \al),u-J(x, \bar \al)\ra-\frac{1}{2\bar \alpha}\| J(x, \bar \al)-x\|^2.
\end{align*}
 Hence, we have
\begin{equation*}
\la x-J(x,\bar\al),J(x,\bar\al)-u\ra\ge \bar\alpha[F(J(x,\bar\al))-F(x)]-\frac{1}{2}\| J(x,\bar\al)-x\|^2.
\end{equation*}
Since $2\la x-J(x,\bar\al),J(x,\bar\al)-u \ra=\|x-u\|^2-\|J(x,\bar\al)-x\|^2-\|J(x,\bar\al)-u\|^2,$
  the latter  implies that
\begin{equation*}
\|x-u\|^2-\|J(x,\bar\al)-u\|^2\ge 2\bar \alpha[F(J(x,\bar\al))-F(x)],
\end{equation*}
which clearly ensures {\rm (ii)}.

Finally,  {\rm (iii)} is a direct consequence of {\rm (i)} with $x=u$. It follows that $J(x,\bar\al)$ belongs to the sublevel set $\{F\le F(x)\}$. By {\bf A2}, $J(x,\bar \al)\in {\rm int} (\dom f)\cap \dom g$ and $f$ is continuously differential at $J(x,\bar\al)$. The proof  is complete.
\end{proof}

Now we recall the forward-backward splitting method with line search proposed by \cite{BTe} as following.

\begin{center}\fbox{\begin{minipage}[b]{\textwidth}
{\bf Forward-backward splitting method with backtracking line search} (FBS method)
\begin{itemize}

\item [    ] {\bf Step 0.} Take $x^0\in {\rm int} (\dom f)\cap \dom g$, $\sigma>0$ and $\theta\in(0,1)$.

\item [    ] \noindent {\bf Step k.} Set
\begin{equation}\label{F-B}
x^{k+1}:=J(x^k,\alpha_{k})=\prox_{\alpha_{k} g}(x^k-\alpha_{k}\nabla
f(x^k))
\end{equation}
with $\al_{-1}:=\sigma$ and
\begin{equation}\label{al}
\alpha_{k} := LS(x^k,\al_{k-1},\theta).
\end{equation}
\end{itemize}
\end{minipage}}
\end{center}
The following result which is a direct consequence of Proposition~{boundary-well} plays the central role in our further study.
\begin{corollary}[Well-definedness of FBS method] \label{prop1}  Let $x^0\in {\rm int} (\dom f)\cap \dom g$. The sequence $ (x^k)_{k\in \NN}$ from FBS method is well-defined and $f$ is differentiable at any $x^k$. Moreover, for all $k\in \NN$ and $x\in\R^n$, we have
\begin{enumerate}[{\rm (i)}]
\item $\|x^k-x\|^2-\|x^{k+1}-x\|^2\ge 2\alpha_{k}\left[F(x^{k+1})-F(x)\right]$.
\item  $F(x^{k+1})-F(x^k)\le -\dsty \frac{1}{2 \alpha_k}\|x^{k+1}-x^k\|^2$.
\end{enumerate}
\end{corollary}
\begin{proof}
Thanks to Proposition~\ref{boundary-well}, $x^k\in {\rm int} (\dom f)\cap \dom g$ and $f$ is differentiable at any $x^k$ inductively. This verifies the well-definedness of $(x^k)_{k\in \NN}$. Moreover, both (i) and (ii) are consequence of (ii) and (iii) from Proposition~\ref{boundary-well} by replacing $u=x$, $x=x^k$, $\bar \al=\al_k$, and $J(x,\bar \alpha)=J(x^k,\al_k)=x^{k+1}$.
\end{proof}

The following result shows that the FBS method with backtracking line search is global convergent without assuming Lipschitz continuity on the gradient $\nabla f$, and so, improves \cite[Theorem~1.2]{BTe}.
A variant of this result for FBS method under different line searches was established in \cite[Theorem~4.2]{YN}.   Here, the proof is
also similar to that of \cite[Theorem~4.2]{YN} by using the well-definedness of $(x^k)_{k\in \NN}$ and two properties in Corollary~\ref{prop1}  and hence, we omit the details.

\begin{theorem}[Global convergence of FBS method] \label{new-cov}  Let $(x^k)_{k\in \NN}$ be the sequence generated from FBS method. The following statements hold:
\begin{enumerate}[{\rm (i)}]
\item  If $S^*\neq\emptyset$ then $(x^k)_{k\in \NN}$ converges to a point in
$S^*$. Moreover, \begin{equation}\label{min*}\lim_{k\to\infty}
F(x^k)=\min_{x\in \R^n}F(x).\end{equation}
\item  If $S^*=\emptyset$ then we have
\begin{equation*}\label{inf}
\lim_{k\to\infty}\|x^k\|=+\infty \quad \mbox{and}
\quad\lim_{k\to\infty} F(x^k)=\inf_{x\in \R^n}F(x).
\end{equation*}
\end{enumerate}
\end{theorem}

Next we present the sublinear convergence for FBS method when the function $f$ is locally Lipschitz continuous. The following proposition tells us that when $f$ is locally Lipschitz continuous, the step size $\al_{k}$ in FBS  method is bounded below by a positive number. The second part of this result coincides with \cite[Remark~1.2]{BTe}.

\begin{proposition}[Boundedness from below for  the step sizes]\label{lema-alpha}
Let $(x^k)_{k\in \NN}$ and $(\alpha_k)_{k\in \NN}$ be the sequences generated from FBS method. Suppose that $S^*\neq\emptyset$ and that the sequence $(x^k)_{k\in \NN}$ is converging to some $x^*\in S^*$. If $\nabla f$ is locally Lipschitz continuous
around $x^*$ with modulus $L$ then there exists some $K\in \NN$ such that
 \begin{equation}\label{in-al}
\al_k\ge \min\left\{\al_K,\frac{\theta}{L}\right\}>0\quad\mbox{for all}\quad k>K.
\end{equation}
Furthermore,   if 
$\nabla f$ is globally Lipschitz continuous on ${\rm int}(\dom f)\cap \dom g$ with uniform modulus $L$ then $\al_k\ge \min\{\sigma,\frac{\theta}{L}\}$ for any $k\in \NN$.
\end{proposition}
\begin{proof}
To justify, suppose that $S^*\neq\emptyset$, the sequence $(x^k)_{k\in \NN}$ is converging to $x^*\in S^*$, and that  $\nabla f$ is locally Lipschitz continuous
around $x^*$ with constant $L>0$.  We find some $\ve>0$ such that
\begin{equation}\label{sL}
\|\nabla f(x)-\nabla f(y)\|\le L\|x-y\|\quad \mbox{for all} \quad x,y\in \mathbbm{B}_\ve(x^*),
\end{equation}
where $\mathbbm{B}_\ve(x^*)$ is the closed ball in $\R^n$ with center  $x^*$ and radius $\ve$. Since $(x^k)_{k\in \NN}$ is converging  to $x^*$, there exists some $K\in \NN$ such that
\begin{equation}\label{ssL}
\|x^k-x^*\|\le \frac{\theta\ve}{2+\theta}<\ve\quad \mbox{for all}\quad k>K
\end{equation}
with $\theta\in(0,1)$ defined in {\bf
Linesearch BT}. We claim that
\begin{equation}\label{newi}
\al_k\ge\min\left\{\al_{k-1},\frac{\theta}{L}\right\}\quad \mbox{ for any }\quad k>K.
\end{equation}
Suppose by contradiction that $\al_k<\min\{\al_{k-1},\frac{\theta}{L}\}$. Then, $\al_k<\al_{k-1}$, and so, the loop in {\bf Linesearch BT} at $(x^k,\al_{k-1})$ needs more than one iteration.  Define
$\hat{\alpha}_k:=\frac{\alpha_k}{\theta}>0$ and $
\hat{x}^k:=J(x^k,\hat{\alpha}_k)$, {\bf Linesearch BT} tells us that
\begin{equation}\label{44}
f(\hat x^k)> f(x^k)+\la \nabla f(x^k),\hat x^k-x^k\ra+\frac{1}{2\hat \al_k}\|x^k-\hat x^k\|^2.
\end{equation}
Furthermore, it follows from Lemma~\ref{lema-prop} that
\[
\|x^{k}-\hat x^{k}\|=\|x^{k}-J(x^{k},\hat \alpha_{k})\|\le \frac{\hat \al_{k}}{\al_{k}}\|x^{k}-J(x^{k},\al_{k})\|= \frac{1}{\theta}\|x^{k}-x^{k+1}\|.
\]
This together with \eqref{ssL} yields
\begin{eqnarray}\label{hat}
\begin{array}{ll}
\|\hat x^k-x^*\|&\disp\le \| \hat x^k- x^k\|+\| x^k-x^*\|\le \frac{1}{\theta}\|x^{k}-x^{k+1}\|+\| x^k-x^*\|\\
&\disp\le \frac{1}{\theta}\cdot \frac{2\theta\ve}{2+\theta}+\frac{\theta\ve}{2+\theta}=\ve.
\end{array}
\end{eqnarray}
Since $x^k, \hat x^k\in \B_\ve(x^*)$ by \eqref{ssL} and \eqref{hat}, we get from \eqref{sL} that
\begin{align*}
f(\hat x^k)-f(x^k)-\la \nabla f(x^k),\hat x^k-x^k\ra&= \int_{0}^1\la\nabla f(x^k+t(\hat x^k-x^k))-\nabla f(x^k),\hat x^k-x^k\ra dt\\
&\le \int_{0}^1tL\|\hat x^k-x^k\|^2dt=\frac{L}{2}\|\hat x^k-x^k\|^2.
\end{align*}
Combining this with \eqref{44} yields $\hat \al_k\ge \frac{1}{L}$ and thus  $\al_k\ge \frac{\theta}{L}$. This is a contradiction.

If there is some $H>K$ with $H\in \NN$ such that  $\al_H>\frac{\theta}{L}$, we get from \eqref{newi} that $\al_k\ge \frac{\theta}{L}$ for all $k\ge H$. Otherwise, $\al_k<\frac{\theta}{L}$ for any $k>K$, which implies that $\al_k=\al_{k-1}=\al_K$ for all $k>K$ due to \eqref{newi} and the nonincreasing property of $(\al_k)_{k\in\NN}$. In both cases we have \eqref{in-al}.

Finally suppose that  $\nabla f$ is globally Lipschitz continuous with modulus $L$  on ${\rm int}(\dom f)\cap \dom g\subset {\rm int} (\dom f)$. By using Proposition \ref{boundary-well}(iii), we can repeat the above proof without concerning $\ve, K$ and replace \eqref{newi} by $\al_k\ge \min\{\sigma,\frac{\theta}{L}\}$.
\end{proof}

The following result showing the complexity $o(k^{-1})$ of FBS method when the function $\nabla f$ is locally Lipschitz continuous,  improves \cite[Theorem~1.1]{BTe}, which only obtains $\mathcal{O}(k^{-1})$ of this method with the stronger assumption that $\nabla f$ is
globally Lipschitz continuous. The proof of this theorem is quite similar to \cite[Theorem~4.3 and Corollary~4.5]{YN}  and so, we omit the details.

\begin{theorem}[Sublinear convergence of FBS method] \label{l-rate}
Let $(x^k)_{k\in \NN}$  be the sequence
generated in FBS method. Suppose that $S^*\neq\emptyset$ and that $\nabla f$ is locally Lipschitz continuous  around any point in $S^*$. Then  we have
\begin{equation}\label{xrate}
\lim_{k\to \infty}k\big[F(x^k)-\min_{x\in \R^n}F(x)\big]=0.
\end{equation}
\end{theorem}

\section{Local linear convergence of forward-backward splitting methods}

In this section, we obtain the local Q-linear convergence for FBS method under a mild assumption of metric subregularity on $\partial F$ and local Lipschitz continuity of $\nabla f$, which is automatic in many problems including Lasso problem and Poisson linear inverse regularized problem. R-linear convergence of FBS method has been recently established under some different assumptions such as  Kurdya-\L ojasiewicz inequality  with order $\frac{1}{2}$ \cite{BNPS}, and the quadratic growth condition \cite{DL}, all of which are equivalent in the convex case; see \cite[Corollary~3.6]{DL} and \cite[Theorem 5]{BNPS}. Our results are close to \cite[Theorem~3.2 and Corollary~3.7]{DL}. However, we focus on the local linear convergence; our proof also suggests a direct way to obtain linear convergence of FBS from the quadratic growth condition \eqref{subre} below without going through the error bound \cite[Definition~3.1]{DL}. The first result is regarding the R-linear convergence of FBS method that will be improved later by Q-linear convergence in Theorem~\ref{TLC3}.

\begin{proposition}[R-linear convergence under metric subregularity]\label{TLC1} Let $(x^k)_{k\in \NN}$ and $(\al_k)_{k\in \NN}$ be the sequences generated from FBS method. Suppose that  $S^*$ is not empty,  $(x^k)_{k\in \NN}$ converges to some $x^*\in S^*$ as in Theorem~\ref{new-cov}, and that $\nabla f$ is locally Lipschitz continuous around $x^*$ with constant $L>0$. If $\partial F=\nabla f+\partial g$ is  metrically subregular at $x^*$ for $0$ with modulus $\kk^{-1}>0$, then there exists some $K\in \NN$ such that
\begin{equation}\label{LC1}
d(x^{k+1};S^*)\le \frac{1}{\sqrt{1+\al\kk}} d(x^k;S^*)\quad\mbox{for all}\quad k>K,
\end{equation}
where $\al:=\min\big\{\frac{\al_K}{2},\frac{\theta}{2L}\big\}$. Consequently, we have
\begin{align}
F(x^{k+1})-\min_{x\in \R^n} F(x)&=\mathcal{O}((1+\al\kk)^{-k})\label{Co1},\\
\|x^{k+1}-x^*\|&=\mathcal{O}((1+\al\kk)^{-\frac{k}{2}}). \label{Co2}
\end{align}
If, in addition, $\nabla f$ is globally Lipschitz continuous on ${\rm int}(\dom f)\cap \dom g$ with constant $L$, $\al$  could be chosen as  $\min\big\{\frac{\sigma}{2},\frac{\theta}{2L}\big\}$, which is independent from $K$.
\end{proposition}
\begin{proof} When  $\partial F$ is metrically subregular at $x^*$ for $0$ with modulus $\kk^{-1}>0$, it follows from Proposition~\ref{SR-grow} that there exists $\ve>0$ such that
\begin{equation}\label{subre}
F(x)-F(x^*)\ge \frac{\kk}{2} d^2 (x;S^*)\quad \mbox{for all}\quad x\in \mathbbm{B}_\ve(x^*).
\end{equation}
Since $(x^k)_{k\in \NN}$ converges to $x^*$ and $\nabla f$ is locally Lipschitz continuous around $x^*$, we find from Proposition~\ref{lema-alpha} some constant $K\in \NN$ such that $\al_k\ge 2\al$ and $x^k\in \B_\ve(x^*)$ for any $k>K$.  Denote the projection of  $a$ onto the set $S^*$ by $\Pi_{S^*}(a)$.
Combining \eqref{subre} with Corollary~\ref{prop1}(i) implies that
\begin{eqnarray}\begin{array}{ll}\label{LC2}
d^2(x^k;S^*)-d^2(x^{k+1};S^*)&\ge \disp\|x^k-\Pi_{S^*}(x^k)\|^2-\|x^{k+1}-\Pi_{S^*}(x^k)\|^2\\
&\ge\al_k [F(x^{k+1})-F(\Pi_{S^*}(x^k))]\\
&\disp\ge2\al[F(x^{k+1})- F(x^*)]\ge\al\kk d^2(x^{k+1};S^*)
\end{array}
\end{eqnarray}
for all $k>K$.  This clearly verifies \eqref{LC1}.

To justify \eqref{Co1}, note from \eqref{LC1} that $d(x^{k};S^*)=\mathcal{O}((1+\al\kk)^{-\frac{k}{2}})$. This together with \eqref{LC2} allows us to find some $M>0$ such that
\[
0\le F(x^{k+1})- F(x^*)\le \frac{1}{2\al}d^2(x^{k};S^*)\le M(1+\al\kk)^{-{k}}\quad  \mbox{for all}\quad  k\in \NN,
\]
which clearly ensures \eqref{Co1}. To verify \eqref{Co2}, we derive from Corollary~\ref{prop1}(ii) that
\[
\|x^k-x^{k+1}\|\le \sqrt{2\al_k[F(x^k)-F(x^{k+1})]}\le \sqrt{2\sigma[F(x^k)-F(x^*)]}\le \sqrt{2\sigma M}(1+\al\kk)^{-\frac{k-1}{2}}.
\]
Since $(x^k)_{k\in \NN}$ converges to $x^*$, it follows from the latter inequality that
\begin{align*}
\|x^{k+1}-x^*\|&=\sum_{j=k+1}^\infty\big(\|x^j-x^*\|-\|x^{j+1}-x^*\|\big)\le \sum_{j=k+1}^\infty\|x^j-x^{j+1}\|\\
&\le \sqrt{2\sigma M}(1+\al\kk)^{-\frac{k}{2}}\sum_{j=0}^\infty(1+\al\kk)^{\frac{-j}{2}}=\sqrt{2\sigma M}(1+\al\kk)^{-\frac{k}{2}}[1-(1+\al\kk)^{-\frac{1}{2}}]^{-1},
\end{align*}
which verifies \eqref{Co2}. To complete, we repeat the above proof with the note  from Proposition~\ref{lema-alpha} that $\al_k\ge \min\{\sigma,\frac{\theta}{L}\}$ when $\nabla f$ is globally Lipschitz continuous on ${\rm int}(\dom f)\cap \dom g$ with constant $L$.
\end{proof}

 In the special case where $g(x)=\delta_X(x)$, the indicator function to a closed convex set $X\subset \R^n$,
the obtained linear convergence of  $(d(x^k;S^*))_{k\in \NN}$ in \eqref{LC1} is close to the \cite[Theorem~12]{NNG} \footnote{Proposition~\ref{TLC1} was presented by the third author at ICCOPT 2016, when he was aware of \cite{NNG} after attending the talk of I. Necoara.}.

 Next, we present the promised Q-linear convergence of the FBS method for both the objective value sequence $(F(x^{k}))_{k\in \NN}$ and the
iterative sequence $(x^k)_{k\in \NN}$, under a general metric subregularity assumption. Easily verifiable sufficient conditions for this
metric subregularity assumption will be provided in Corollary \ref{Ma} and Section 5 later. We also point out that Q-linear convergence on the objective value sequence $(F(x^{k}))_{k\in \NN}$  has been discovered in
the recent papers \cite{BNPS,LP} under the assumption that $F$ satisfies K\L\, inequality with the exponent $\frac{1}{2}$ at $x^*$, which is equivalent to the metric subregularity of $\partial F$ at $x^*$ for $0$; however, the Q-linear convergence of $(x^k)_{k\in \NN}$ obtained here is new.

\begin{theorem}[Q-linear convergence under metric subregularity]\label{TLC3}  Let $(x^k)_{k\in \NN}$ and $(\al_k)_{k\in \NN}$ be the sequences generated from FBS method. Suppose that the solution set $S^*$ is not empty,  $(x^k)_{k\in \NN}$ converges to some $x^*\in S$, and that $\nabla f$ is locally Lipschitz continuous around $x^*$ with constant $L>0$. If $\partial F=\nabla f+\partial g$ is metrically subregular at $x^*$ for $0$ with modulus $\kk^{-1}>0$, there exists $K\in \NN$ such that
\begin{align}
\|x^{k+1}-x^*\|&\le \frac{1}{\sqrt{1+\frac{\al\kk}{4}}}\|x^k-x^*\|\label{LC4}\\
|F(x^{k+1})-F(x^*)|&\le \frac{\sqrt{1+\frac{\al\kk}{4}}+1}{2\sqrt{1+\frac{\al\kk}{4}}} |F(x^{k})-F(x^*)|\label{LC5}
\end{align}
for any $k>K$, where $\al:=\min\big\{\frac{\al_K}{2},\frac{\theta}{2L}\big\}$.

 If, in addition, $\nabla f$ is globally Lipschitz continuous  on ${\rm int}(\dom f)\cap \dom g$ with constant $L>0$, $\al$  could be chosen as  $\min\big\{\frac{\sigma}{2},\frac{\theta}{2L}\big\}$.
\end{theorem}
\begin{proof}
Since $\partial F=\nabla f+\partial g$ is  metrically subregular at $x^*$ for $0$ with the modulus $\kk^{-1}>0$, we also have \eqref{subre}. This together with Corollary~\ref{prop1}(i) gives us that
\begin{equation}\label{im1}
\|x^k-x\|^2-\|x^{k+1}-x\|^2\ge 2\al [F(x^{k+1})-F(x^*)]\ge \al\kk d^2(x^{k+1};S^*) \; \mbox{for all}\; x\in S^*,
\end{equation}
when $k>K$ for some large $K\in \NN$. Moreover, for any $r>k>K$ we get  that
\[
\|x^r-\Pi_{S^*}(x^{k+1})\|\le \|x^{k+1}-\Pi_{S^*}(x^{k+1})\|=d(x^{k+1};S^*).
\]
  Taking $r\to \infty$ gives us that $\|x^*-\Pi_{S^*}(x^{k+1})\|\le d(x^{k+1};S^*)$. It follows that
\[
\|x^{k+1}-x^*\|\le \|x^{k+1}-\Pi_{S^*}(x^{k+1})\|+\|\Pi_{S^*}(x^{k+1})-x^*\|\le 2 d(x^{k+1}; S^*).
\]
This together with \eqref{im1} implies that
\[
\|x^k-x^*\|^2\ge \|x^{k+1}-x^*\|^2+\frac{\al\kk}{4}\|x^{k+1}-x^*\|^2=\left(1+\frac{\al\kk}{4}\right)\|x^{k+1}-x^*\|^2,
\]
which clearly verifies \eqref{LC4}.

To see the second conclusion, we note from \eqref{LC4} that
\begin{equation}\label{beta0}
\|x^{k+1}-x^k\|\ge \|x^{k}-x^*\|-\|x^{k+1}-x^*\|\ge \beta (\|x^{k+1}-x^*\|+\|x^k-x^*\|)
\end{equation}
 with $\beta:=\frac{\sqrt{1+\al\kk/4}-1}{\sqrt{1+\al\kk/4}+1}$ for  $k>K$ sufficiently large.
We derive from this, Corollary~\ref{prop1}(ii), and \eqref{beta0} that
\begin{align*}
F(x^{k})-F(x^{k+1})&\ge \frac{1}{2\al_k}\|x^{k+1}-x^k\|^2\ge \frac{1}{2\al_k}(\|x^k-x^*\|-\|x^{k+1}-x^*\|)^2\\
&\ge \frac{\beta}{2\al_k} (\|x^k-x^*\|-\|x^{k+1}-x^*\|)(\|x^k-x^*\|+\|x^{k+1}-x^*\|)\\
&\ge \frac{\beta}{2\al_k} (\|x^k-x^*\|^2-\|x^{k+1}-x^*\|^2).
\end{align*}
Hence, we get from Corollary~\ref{prop1}(i)  that
\begin{align*}
F(x^{k})-F(x^{k+1})\ge \beta[F(x^{k+1})-F(x^*)].
\end{align*}
It follows that
$
F(x^k)-F(x^*)\ge (1+\beta)[F(x^{k+1})-F(x^*)],
$
which clarifies \eqref{LC5}.

The last statement can be obtained similarly to the preceding proposition.
 \end{proof}

 Next, we show that a sharper $Q$-linear convergence rate
of $(x^k)_{k\in \NN}$ and $(F(x^k))_{k\in \NN}$ can be obtained under a stronger assumption: strong metric subregularity.

\begin{corollary}[Sharper Q-linear convergence rate under strong metric subregularity] \label{TLC2} Let $(x^k)_{k\in \NN}$ and $(\al_k)_{k\in \NN}$ be the sequences generated from FBS method. Suppose that the solution set $S^*$ is not empty,  $(x^k)_{k\in \NN}$ converges to some $x^*\in S^*$, and that $\nabla f$ is locally Lipschitz continuous around $x^*$ with constant $L>0$. If $\partial F=\nabla f+\partial g$ is strongly metrically subregular at $x^*$ for $0$ with modulus $\kk^{-1}>0$, then $x^*$ is the unique solution to problem \eqref{prob}. Moreover, there exists some $K\in \NN$ such that for any $k>K$ we have
\begin{align}
\|x^{k+1}-x^*\|&\le\frac{1}{\sqrt{1+\al\kk}}\|x^k-x^*\|\label{Co3}\\
|F(x^{k+1})-F(x^*)|&\le \frac{\sqrt{1+\al\kk}+1}{2\sqrt{1+\al\kk}} |F(x^{k})-F(x^*)| \label{Co4}
\end{align}
with $\al:=\min\big\{\frac{\al_K}{2},\frac{\theta}{2L}\big\}$.

Additionally, $\nabla f$ is globally Lipschitz continuous  on ${\rm int}(\dom f)\cap \dom g$ with constant $L>0$, $\al$ above could be chosen as  $\min\big\{\frac{\sigma}{2},\frac{\theta}{2L}\big\}$.

\end{corollary}
\begin{proof}
If $\partial F=\nabla f+\partial g$ is strongly metrically subregular at $x^*$ for $0$ with modulus $\kk^{-1}>0$, $x^*$ is an isolated point of $\partial F^{-1}(0)=S^*$. Since $S^*$ is a closed convex set, we have $S^*=\{x^*\}$. Thus \eqref{Co3} is a direct consequence of \eqref{LC1}. To verify \eqref{Co4},  we note from \eqref{Co3} that
\begin{equation}\label{beta}
\|x^{k+1}-x^k\|\ge \|x^{k}-x^*\|-\|x^{k+1}-x^*\|\ge \beta (\|x^{k+1}-x^*\|+\|x^k-x^*\|)
\end{equation}
with $\beta=\frac{\sqrt{1+\al\kk}-1}{\sqrt{1+\al\kk}+1}$ for any $k>K$ sufficiently large. The proof of Q-linear convergence of $(F(x^{k}))_{k\in\NN}$ in \eqref{Co4}
can be obtained similarly as in Theorem \ref{TLC3} by using \eqref{beta} instead of \eqref{beta0}.
\end{proof}

 The assumption that $\partial F$ is metrically subregular in above results is automatic for a broad class of so-called {\em piecewise linear-quadratic functions} \cite[Definition~10.20]{rw} defined below.

 \begin{definition}[convex piecewise linear-quadratic functions] A function $h\in \Gamma_0(\R^n)$ is called {\em convex piecewise linear-quadratic} if $\dom h$ is a union of finitely many polyhedral sets, relative to each of which $h(x)$ is given the expression of the form $\frac{1}{2}\la x,Ax\ra+\la b,x\ra+c$ for some scalar $c\in \R$, vector $b\in  \R^n$ and a symmetric positive semi-definite $A\in \R^{n\times n}$.
\end{definition}

If $F=f+g$ is convex piecewise linear-quadratic function, it is known from \cite[Proposition~12.30]{rw} that the set-valued mapping $\partial F$ is polyhedral and thus is metrically subregular at any point $\ox\in \dom \partial F$ for any $\bar v\in \partial F(\ox)$ by Proposition~\ref{MSP}. This observation together with Theorem~\ref{TLC1}  tells us the local R-linear convergence of FBS for convex piecewise linear-quadratic functions. This fact has been obtained and discussed before in \cite{DL,LP,ZS}.  Our following result advances it with the Q-linear convergence and the uniform convergence rate.

 \begin{corollary}[Local linear convergence for piecewise linear-quadratic functions]\label{Ma} Let $(x^k)_{k\in\NN}$ and $(\al_k)_{k\in\NN}$ be the sequences generated from FBS method. Suppose that $F=f+g$ is a convex piecewise linear-quadratic function, the solution set $S^*$ is nonempty, and that $\nabla f$ is locally Lipschitz continuous around any point in $S^*$. Then the sequences $(x^k)_{k\in\NN}$ and $(F(x^k))_{k\in\NN}$ are globally convergent to some optimal solution and optimal value respectively with local Q-linear rates.

Furthermore, if $\nabla f$ is globally Lipschitz continuous on ${\rm int}(\dom f)\cap \dom g$, $(x^k)_{k\in\NN}$ and $(F(x^k))_{k\in\NN}$ are globally convergent to some optimal solution and optimal value, respectively,  with uniform local linear rates that do not depend on the choice of the initial point $x^0$.
 \end{corollary}
\begin{proof}
Suppose that the sequence $(x^k)_{k\in\NN}$ converges to some $x^*\in S^*$ by Theorem~\ref{new-cov}. Since $F$ is a convex piecewise linear-quadratic function, the graph of $\partial F$ is polyhedral and thus it is metrically subregular at $x^*$  for $0$ with a uniform rate $\kk^{-1}>0$, which does not depend on the choice of $(x^k)_{k\in\NN}$ and $x^*$ by Proposition~\ref{MSP}.  By Theorem~\ref{TLC3}, we have  $(x^k)_{k\in\NN}$ and $(F(x^k))_{k\in\NN}$ are locally convergent to some $x^*$  and the optimal value $F(x^*)$, respectively, with Q-linear rate.

To complete the proof, suppose that $\nabla f$ is globally Lipschitz continuous on  ${\rm int}(\dom f)\cap \dom g$ with constant $L$. It follows from the last part of Theorem~\ref{TLC1} that  $\al$ could be chosen as $\min\{\sigma,\frac{\theta}{L}\}$. Since the metric subregularity modulus of $\partial F$ is uniform as discussed above, the   linear rate in Theorem~\ref{TLC1} is independent from the choice of initial points.
\end{proof}

\begin{Remark}\label{rem}{\rm
It is worth noting that all the assumptions in Corollary~\ref{Ma} on initial data hold automatically in many important classes optimization problems in practice including the Tikhonov regularization, wavelet-based regularization, $\ell_1$ regularization, $\ell_\infty$ regularization least square problems; see further discussions about using FBS method in these problems in \cite{BTe, HYZ, LFP, TBZ}. Let us discuss a bit here about Lasso problem \eqref{Pros}.
It is easy to see that  $F_2$ in \eqref{Pros} is a convex piecewise linear-quadratic function. Moreover, the function $f(x)=\frac{1}{2}\|Ax-b\|^2$  has the gradient $\nabla f(x)=A^T(Ax-b)$ that is globally Lipschitz continuous on $\R^n$.  FBS method for problem \eqref{Pros} is also called {\em iterative shrinkage thresholding algorithm} (ISTA) \cite{BTe} via the shrinkage thresholding mapping $\prox_{\mu\|\cdot\|_1}$.  Recently, Tao-Boley-Zhang \cite[Theorem~5.9]{TBZ} shows that the ISTA iteration eventually linearly convergent  provided that \eqref{Pros} has a unique solution that satisfies a {\em strict complementarity} condition.  Our Corollary~\ref{Ma} tells that not only ISTA iteration but also their functional iteration  eventually reach the stage of linear convergence without adding any extra condition. Moreover, the linear rate is uniform and computable; see our Section~6.3 for computing this rate and also the global linear convergence of ISTA.

}
\end{Remark}

\section{Linear convergence of forward-backward splitting method in some structured optimization problems}

\subsection{Poisson linear inverse problem}
This subsection devotes to the study of the {\em eventually} linear convergence of FBS when solving the following  standard Poisson regularized problem \cite{C,VSK}
\begin{equation}\label{PKL}
\min_{x\in \R^n_+}\sum_{i=1}^mb_i\log\frac{b_i}{(Ax)_i}+(Ax)_i-b_i,
\end{equation}
where  $A\in \R^{m\times n}_+$ is an $m\times n$  matrix with nonnegative entries and  nontrivial rows, and $b\in \R^m_{++}$ is a positive vector. This problem is usually used to recover a signal $x\in \R^n_+$ from the  measurement $b$ corrupted by Poisson noise satisfying $Ax\simeq b$. The problem \eqref{PKL} could be written in term of \eqref{prob} in which
\begin{equation}\label{fgh}
f(x):=h(Ax),\qquad g(x)=\delta_{\R_+^n}(x), \qquad \mbox{and}\qquad F_3(x):=h(Ax)+g(x),
\end{equation}
where $h$ is the Kullback-Leibler divergence defined by
\begin{eqnarray}\label{H}
h(y)=\left\{\begin{array}{ll} \disp\sum_{i=1}^mb_i\log\dfrac{b_i}{y_i}+y_i-b_i\quad &\mbox{if}\quad y\in \R^m_{++},\\
 +\infty \quad &\mbox{if}\quad y\in \R^m_+\setminus\R^m_{++}.\end{array}\right.
\end{eqnarray}
Note from \eqref{fgh} and \eqref{H} that $\dom f=A^{-1}(\R^m_{++})$, which is an open set. Moreover, since $A\in \R^{m\times n}_+$, we have $\dom f\cap \dom g={\rm int} (\dom f)\cap \dom g=A^{-1}(\R^m_{++})\cap \R^n_+\neq \emptyset$ and $f$ is continuously differentiable  at any point on $ \dom f\cap \dom g$. The standing assumptions {\bf A1} and {\bf A2} are satisfied for Problem \eqref{PKL}. Moreover, since the function $F_3$ is bounded below and coercive, the optimal solution set to problem \eqref{PKL} is always nonempty.

It is worth noting further that  $\nabla f$ is locally Lipschitz continuous at any point ${\rm int}(\dom f)\cap \dom g$ but not globally Lipschitz continuous on ${\rm int}(\dom f)\cap \dom g$. Our  Theorem~\ref{l-rate} is applicable to solving \eqref{PKL} with global convergence  rate $o(\frac{1}{k})$. In the recent work \cite{BBT}, a new algorithm rather close to FBS was designed with applications to solving \eqref{PKL}. However, the theory developed in \cite{BBT} could not guarantee the global convergence of their optimal sequence $(x^k)_{k\in\NN}$  when solving \eqref{PKL}, since one of their assumptions on the closedness of the domain of their auxiliary Legendre function in \cite[Theorem~2]{BBT} is not satisfied.
Our intent in this subsection is to reveal the Q-linear convergence of our method when solving \eqref{PKL} in the sense of Theorem~\ref{TLC3}.  In order to do so, we need to verify the metric subregularity of $\partial F_3$ at any optimal minimizer for $0$, or the second-order growth condition of $F_3$. Note further that the Kullback-Leibler divergence  $h$ is not strongly convex and $\nabla f$ is not globally Lipschitz continuous; hence, standing assumptions in \cite{DL} are not satisfied. Proving the metric subregularity of $\partial F_3$ at an optimal solution via the approach of \cite{DL} needs to be proceeded with caution.

\begin{lemma}\label{LKL}    Let $\ox$ be an optimal solution to problem \eqref{PKL}. Then for any $R>0$, we have
\begin{equation}\label{GP}
F_3(x)-F_3(\ox)\ge \nu d^2(x;S^*)\quad \mbox{for all}\quad x\in \B_R(\ox)
\end{equation}
with some positive constant $\nu$. Consequently, $\partial F_3$ is metrically subregular with  at $\ox$ for $0$ with modulus $\nu^{-1}$.
\end{lemma}	
\begin{proof}
Pick any $R>0$ and $x\in \B_R(\ox)$. We only need to prove \eqref{PKL} for the case  that $x\in \dom F_3\cap\B_R(\ox)$, i.e., $x\in A^{-1}(\R^n_{++})\cap\R^n_+\cap\B_R(\ox)$. Note that
\[
\nabla f(x)=\sum_{i=1}^m\big[1-\frac{b_i}{\la a_i,x\ra}\big]a_i\quad \mbox{and}\quad \la\nabla^2f(x)d,d\ra=\sum_{i=1}^m b_i \frac{\la a_i,d\ra^2}{\la a_i,x\ra^2}\quad \mbox{for all}\quad d\in \R^n,
\]
where $a_i$ is the i-th row of $A$.
Define $\oy:=A\ox$,  for any $x,u\in \B_R(\ox)\cap\dom f$ we have $[x,u]\subset \B_R(\ox)\cap\dom f$ and obtain from the mean-value  theorem that
\begin{align*}
f(x)-f(u)-\la\nabla f(u),x-u\ra&\disp=\frac{1}{2}\int_{0}^1\la\nabla^2f(u+t(x-u))x-u,x-u\ra dt\\
&\disp= \frac{1}{2}\int_{0}^1\sum_{i=1}^m b_i \frac{\la a_i,x-u\ra^2}{\la a_i,u+t(x-u)\ra^2}dt
\\
&\disp\ge \frac{1}{2}\int_{0}^1\sum_{i=1}^m b_i  \frac{\la a_i,x-u\ra^2}{[\la a_i,\ox\ra+\|a_i\|(\|u-\ox\|+t\|x-u\|)]^2}dt\\
&\disp\ge \frac{1}{2}\sum_{i=1}^m  \frac{b_i}{[\la a_i,\ox\ra+3\|a_i\|R]^2}\la a_i,x-u\ra^2.
\end{align*}	
Similarly, we have
\begin{equation}\label{ux}
f(u)-f(x)-\la\nabla f(x),u-x\ra\ge \frac{1}{2}\sum_{i=1}^m  \frac{b_i}{[\la a_i,\ox\ra+3\|a_i\|R]^2}\la a_i,u-x\ra^2\;\; \mbox{for}\;\;x,u\in  \B_R(\ox)\cap\dom f.
\end{equation}
Adding the above two inequalities gives us that
\begin{equation}\label{ux2}
\la\nabla f(x)-\nabla f(u),x-u\ra\ge  \sum_{i=1}^m  \frac{b_i}{[\la a_i,\ox\ra+3\|a_i\|R]^2}\la a_i,x-u\ra^2\quad \mbox{for all}\quad x,u\in  \B_R(\ox)\cap\dom f.
\end{equation}
We claim that the optimal solution set $S^*$ to problem \eqref{PKL} satisfies that
\begin{equation}\label{inS}
S^*=A^{-1}(\oy)\cap (\partial g)^{-1}(-\nabla f(\ox))\quad \mbox{with}\quad \oy=A\ox.
\end{equation}
Pick another optimal solution $\ou\in S^*$, we have $\ou_t:=\ox+t(\ox-\ou)\in S^*\subset \dom f$ for any $t\in [0,1]$ due to the convexity of $S^*$. By choosing $t$ sufficiently small, we have $\ou_t\in \B_R(\ox)\cap\dom f$. Note further that  $-\nabla f(\ou_t)\in \partial g(\ou_t)$ and $-\nabla f(\ox)\in \partial  g(\ox)$. Since $\partial g$ is a monotone operator, we obtain that
\[
0\ge \la \nabla f(\ox)-\nabla f(\ou_t),\ox-\ou_t\ra.
\]
This together with \eqref{ux2} tells us that $\la a_i, \ox-\ou_t\ra=0$ for all $i=1,\ldots,m$. Hence $A\ox=A\ou=\oy$ for any $\ou\in S^*$, which also implies that
\begin{equation}\label{li}
\nabla f(\ou)=A^T\nabla h(A\ou)=A^T\nabla h(A\ox)=\nabla f(\ox).
\end{equation}
This verifies the inclusion ``$\subset$'' in \eqref{inS}. The opposite inclusion is trivial. Indeed, take any $u$ satisfying that $Au=\oy$ and $-\nabla f(\ox)\in \partial g(u)$, similarly to \eqref{li} we have $-\nabla f(u)=-\nabla f(\ox)\in \partial g(u)$. This shows that $0\in  \nabla f(u)+\partial g(u)$, i.e., $u\in S^*$. The proof for equality \eqref{inS} is completed.

Note from \eqref{inS} that the optimal solution set $S^*$ is a polyhedral with the following format
\[
S^*=\{u\in \R^n\,|\; Au=\oy=A\ox, \la\nabla f(\ox),u\ra=0, u\in \R^n_+\}
\]
due to the fact that $(\partial g)^{-1}(-\nabla f(\ox))=\{u\in \R^n_+\,|\; \la \nabla f(\ox),u\ra=0= \la\nabla f(\ox),\ox\ra\}. $ Thanks to the Hoffman's lemma, there exists a constant $\gamma>0$ such that
\begin{equation}\label{Hoff}
d(x;S^*)\le \gamma(\|Ax-A\ox\|+ |\la\nabla f(\ox),x-\ox\ra|)\quad \mbox{for all}\quad x\in \R^n_+.
\end{equation}
Moreover, for any $x\in \B_R(\ox)\cap\R^n_+$,   \eqref{ux} tells us that
\begin{eqnarray}\label{we}\begin{array}{ll}
f(x)-f(\ox)-\la\nabla f(\ox),x-\ox\ra&\disp\ge \frac{1}{2}\min_{1\le i\le m}\Big[\frac{b_i}{[\la a_i,\ox\ra+3\|a_i\|R]^2}\Big]\|Ax-A\ox\|^2.
\end{array}\end{eqnarray}
This implies that
\begin{align*}
f(x)-f(\ox)&\disp\ge \frac{1}{2}\min_{1\le i\le m}\Big[\frac{b_i}{[\la a_i,\ox\ra+3\|a_i\|R]^2}\Big]\|Ax-A\ox\|^2+ \la\nabla f(\ox),x-\ox\ra\\
&\disp\ge \frac{1}{2}\min_{1\le i\le m}\Big[\frac{b_i}{[\la a_i,\ox\ra+3\|a_i\|R]^2}\Big]\|Ax-A\ox\|^2+\frac{1}{\|\nabla f(\ox)\|\cdot\|x-\ox\|}\la\nabla f(\ox),x-\ox\ra^2\\
&\ge \disp \min\Big\{\frac{1}{2}\min_{1\le i\le m}\Big[\frac{b_i}{[\la a_i,\ox\ra+3\|a_i\|R]^2}\Big], \frac{1}{\|\nabla f(\ox)\|R}\Big\}[\|Ax-A\ox\|^2+\la\nabla f(\ox),x-\ox\ra^2]\\
&\ge \disp \frac{1}{2}\min\Big\{\frac{1}{2}\min_{1\le i\le m}\Big[\frac{b_i}{[\la a_i,\ox\ra+3\|a_i\|R]^2}\Big], \frac{1}{\|\nabla f(\ox)\|R}\Big\}\big[\|Ax-A\ox\| +|\la\nabla f(\ox),x-\ox\ra| \big]^2\\
&\disp\ge \frac{1}{2\gamma^2} \min\Big\{\frac{1}{2}\min_{1\le i\le m}\Big[\frac{b_i}{[\la a_i,\ox\ra+3\|a_i\|R]^2}\Big], \frac{1}{\|\nabla f(\ox)\|R}\Big\}d^2(x;S^*),
\end{align*}
where the fourth inequality follows from the elementary inequality that $\frac{(a+b)^2}{2} \le a^2+b^2$ with $a,b \ge 0$, and the last inequality is from \eqref{Hoff}. This clearly ensures \eqref{GP}. The second part of the lemma is a consequence of Proposition \ref{SR-grow}.\end{proof}

When applying FBS to solving problem \eqref{PKL}, we have
\begin{equation}\label{FBKL}
x^{k+1}=\mathbb{P}_{\R^n_+}\left(x^k-\al_k\sum_{i=1}^m\big[1-\frac{b_i}{\la a_i,x^k\ra}\big]a_i\right)\quad \mbox{with}\quad x^0\in A^{-1}(\R^n_{++})\cap\R^n_+,
\end{equation}
where $\al_k$ is determined from the Beck-Teboulle's line search and $\mathbb{P}_{\R^n_+}(\cdot)$ is the projection mapping to $\R^n_+$. Due to Corollary~\ref{prop1}, all $x^k$ are well-defined and $F_3(x^k)$ are finite.

\begin{corollary}\label{CoroKL} {\rm (Q-linear convergence of method \eqref{FBKL})}  Let $(x^k)_{k\in\NN}$ be the sequence generated from \eqref{FBKL} with $x^0\in A^{-1}(\R^n_+)\cap\R^n_+$ for solving the
Poisson regularized problem \eqref{PKL}. Then the sequences $(x^k)_{k\in\NN}$  and $(F_3(x^k))_{k\in\NN}$ are  Q-linearly convergent to an optimal solution and the optimal value  to \eqref{PKL} respectively.
\end{corollary}
\begin{proof} Since both functions $f$ and $g$ in problem \eqref{PKL} satisfy our standing assumptions {\bf A1} and {\bf A2}, and problem \eqref{PKL} always  has optimal solutions, the sequence $(x^k)_{k\in\NN}$ converges to an optimal solution $\ox$ to problem \eqref{PKL} by Theorem~\ref{new-cov}.
Moreover, it follows from Lemma~\ref{LKL} that $\partial F_3$ is metrically subregular at $\ox$ for $0$.. Since $\nabla f$ is locally Lipschitz continuous around $\ox$, the combination of Theorem~\ref{TLC3} and Lemma~\ref{LKL} tells us that $(x^k)_{k\in\NN}$ is Q-linearly convergent to $\ox$.
\end{proof}

By using this approach, it is similar to show that quadratic growth condition in Lemma~\ref{LKL} is also valid for the following problem
\begin{equation}\label{PKL2}
\min_{x\in \R^n_+}\sum_{i=1}^mb_i\log\frac{b_i}{(Ax)_i}+(Ax)_i-b_i+\mu\, k(x),
\end{equation}
where $k(x)$ is either the $\ell_1$ norm $\|x\|_1$ \cite{BBT}  or the discrete total variation  $TV(x)$ \cite{Bon12}, while $\mu>0$ is the penalty parameter (due to the polyhedral property of $k(x)$.) In particular, when $k(x)=\|x\|_1$, the FBS method for solving \eqref{PKL2} is practical by modifying the function $f(x)$ in \eqref{fgh} to $h(Ax)+\la e,x\ra$ with $e=(1,1,\ldots,1)\in \R^n$. This together with Corollary~\ref{CoroKL} clearly shows that FBS \eqref{PKL} solves the Poisson inverse problem with sparse regularization \cite{BBT} with linear rate.

\subsection{Forward-backward splitting method under partial smoothness}
This subsection is motivated from the recent work \cite{LFP,LFP2} of Liang-Fadili-Peyr\'e in which they study the local linear convergence of FBS method under additional assumptions of partial smoothness on the (possibly nonsmooth) function $g$ that allows them to cover a wide range of important polyhedral/nonpolyhedral optimization problems. The main result of \cite{LFP} is to obtain the linear convergence of FBS iteration under a nondegeneracy assumption and a local strong convexity one. In this section, we revisit the problem in \cite{LFP,LFP2} and obtain some improvements such as   local Q-linear convergence of FBS iteration can be obtained under weaker conditions.

Let us proceed by providing some useful notions mainly used in this section. For any set $\Omega\subset\R^n$, we denote ${\rm ri}\, \Omega$, ${\rm aff}\,\Omega$, ${\rm par}\, \Omega$ by the relative interior, the affine hull, and the subspace parallel to $\Omega$, respectively. Given $x\in \dom g$, define $
T_x:=\left[{\rm par}\,(\partial g(\ox))\right]^\perp.
$ Next we recall the definition of partial smoothness of functions introduced by Lewis \cite{L} with a slight modification for convex functions as in \cite{LFP}. The class of partly smooth functions is broad, including, in particular, convex piecewise linear functions and  spectral functions.

\begin{definition}[Partial smoothness] The convex function $g\in \Gamma_0(\R^n)$ is $\mathcal{C}^2$-{\em partly smooth} at $\ox$ relative to a set $\mathcal{M}$ containing $\ox$ if
\begin{enumerate}
\item {\rm (Smoothness)} $\mathcal{M}$ is a $\mathcal{C}^2$-manifold around $\ox$ and $g$ restricted to $\mathcal{M}$ is $\mathcal{C}^2$ around $\ox$.

\item {\rm (Sharpness)} The tangent space $T_\mathcal{M}(\ox)$ is $T_{\ox}$.

\item {\rm (Continuity)} The subgradient  mapping $\partial g$ is continuous at $\ox$ relative to $\mathcal{M}$.
\end{enumerate}
The class of partly smooth and lower semi-continuous convex functions at $\ox$ relative to $\mathcal{M}$ defined above is denoted by ${\rm PS}_{\ox}(\mathcal{M})$.

\end{definition}

We also define here the so-called {\em covariant Hessian} of a partly smooth function \cite[Definition~2.11]{LZ} as follows. Its computation via the manifold $\mathcal{M}$ and the representation function of $g$ on $\mathcal{M}$ can be found in  \cite{MM}.

\begin{definition}[Covariant Hessian]\label{Cov} Let $g\in \Gamma_0(\R^n)$ be   $\mathcal{C}^2$-{\em partly smooth}  at $\ox$ relative to $\mathcal{C}^2$ manifold $\mathcal{M}$ containing $\ox$. The {\em covariant Hessian},
$\nabla^2_\mathcal{M}g(\ox):T_\mathcal{M}(\ox)\times T_\mathcal{M}(\ox)\to \R$ is the unique self-adjoint and bilinear map satisfying
\[
\la \nabla^2_\mathcal{M} g(\ox)u,u\ra=\frac{d^2}{dt^2}g(\Pi_\mathcal{M}(\ox+tu))\Big|_{t=0}\quad \mbox{for all}\quad u\in T_\mathcal{M}(\ox),
\]
which is known to be well-defined.
\end{definition}

The following result gives a characterization of strong metric subregularity for $\partial F$. Its root can be found from the recent result of Lewis-Zhang \cite[Theorem~6.3]{LZ}.

\begin{proposition}[Characterizations of  strong metric subregularity: partial smoothness cases]\label{Char1} Let $x^*\in S^*$ be an optimal solution to problem \eqref{prob}. Suppose that $f$ is $\mathcal{C}^2$ around $x^*$ and $g$ is  $\mathcal{C}^2$-partly smooth at the point $x^*$ relative to the $\mathcal{C}^2$ manifold $\mathcal{M}$. Suppose further that $-\nabla f(x^*)\in {\rm ri}\, \partial g(x^*)$. The following statements are equivalent:
\begin{enumerate}[{\rm(i)}]

\item  $\partial F$ is strongly metrically subregular at $x^*$ for $0$.

\item $x^*$ is a tilt-stable local minimizer to $F$ in the sense of
 {\rm Proposition~\ref{tilt}{\rm (iii)}}.

\item The following positive-definite condition holds:
\begin{equation}\label{tilt1}
\la(\nabla^2 f(x^*)+\nabla^2_\mathcal{M}g(x^*))u,u\ra>0\quad \mbox{for all}\quad u\in T_\mathcal{M}(x^*)\neq\{0\}.
\end{equation}
\end{enumerate}
Moreover, if {\rm (iii)} is fulfilled then $\partial F$ is strongly metrically subregular at $x^*$ for $0$ with any modulus $\kk>\mu^{-1}$, where $\mu$ is defined by
\begin{equation}\label{c}
\mu:=\min\left\{\frac{\la (\nabla^2 f(x^*)+\nabla^2_\mathcal{M}g(x^*))u,u\ra}{\|u\|^2}|\; u\in T_\mathcal{M}(x^*)\right\}>0
\end{equation}
with the convention $\frac{0}{0}=\infty$.
\end{proposition}
\begin{proof}
Let us start with the implication [(i)$\Longrightarrow$(ii)]. Suppose that $\partial F$ is strongly  metrically subregular at $x^*$ for $0$ with modulus $\kk^{-1}>0$. It follows from Proposition~\ref{SR-grow} that there is some neighborhood $U$ of $x^*$ such that
\begin{equation}\label{qua}
F(x)\ge F(x^*)+\frac{1}{2\kk}\|x-x^*\|^2\quad \mbox{for all}\quad x\in U.
\end{equation}
Note that $0\in {\rm ri}\, \partial F(x^*)$. This together \eqref{qua} means that $x^*$ is a {\em strong critical point} of $F$ relative $\mathcal{M}$ in the sense of \cite[Definition~3.4]{LZ}. Since $F=f+g$ is $\mathcal{C}^2$-partly smooth at the point $x^*$ relative to the $\mathcal{C}^2$ manifold $\mathcal{M}$ due to \cite[Corollary~4.7]{L}, we get from \cite[Theorem~6.3]{LZ} that $x^*$ is a tilt stable local minimum of $f$. The converse implication [(ii)$\Longrightarrow$(i)] is trivial by Proposition~\ref{tilt}.

 Applying \cite[Theorem~6.1 and Theorem~5.3]{LZ} to the function $F$, we also have the equivalence of (ii) to the following condition
\begin{equation*}
\la \nabla^2_\mathcal{M}F(x^*)u+v,u\ra>0\quad \mbox{for all}\quad u\in T_\mathcal{M}(x^*)\setminus\{0\}, v\in N_\mathcal{M}(x^*),
\end{equation*}
where $N_\mathcal{M}(x^*)$ is the normal cone to $\mathcal{M}$ at $x^*$, which is the orthogonal dual of the tangent cone $T_\mathcal{M}(x^*)$ in this case, since $T_\mathcal{M}(x^*)$ is indeed a subspace.

Since  $f$ is $\mathcal{C}^2$ around $x^*$, we have $\nabla^2_\mathcal{M} F(x^*)=\nabla^2 f(x^*)+\nabla^2_\mathcal{M} g(x^*)$. Moreover, $\la v,u\ra= 0$ for all $u\in T_\mathcal{M}(x^*)\setminus\{0\}$ and $v\in N_\mathcal{M}(x^*)$, the latter is equivalent to \eqref{tilt1}. We derive the equivalence between (ii) and (iii).

 Finally let us prove the connection between $\mu$ in \eqref{c} and the strong metric subregular modulus of $\partial F$ at $x^*$ for $0$ in (i).  Suppose that \eqref{tilt1} holds. Then it follows from \cite[Theorem~3.6]{MN1} that $x^*$ is a tilt-stable local minimizer to $F$ with any modulus $\kk>\mu^{-1}$. By Proposition~\ref{tilt},  we have $\partial F$ is strongly metrically subregular at $x^*$  for $0$ with such modulus $\kk$. The proof is complete.
\end{proof}

Recently, to prove the local linear convergence of FBS method when $g$ is partly smooth, \cite{LFP} supposes the nondegeneracy condition $-\nabla f(x^*)\in {\rm ri}\, (\partial g(x^*))$  together with the following assumption
\begin{equation}\label{tilt2}
\la\nabla^2 f(x^*)u,u\ra\ge c\|u\|^2\quad \mbox{for all}\quad u\in T_\mathcal{M}(x^*)
\end{equation}
for some $c>0$. We show next that this condition is stronger than \eqref{tilt1} and thus also guarantee the strong metric subregularity of $\partial F$ at $x^*$ for $0$.

\begin{corollary}[Sufficient condition for strong metric subregularity]\label{Suf} Let $x^*\in S^*$ be an optimal solution. Suppose that $f$ is $\mathcal{C}^2$ around $x^*$ and $g$ is  $\mathcal{C}^2$-partly smooth at the point $x^*$ relative to the $\mathcal{C}^2$-manifold $\mathcal{M}$. Suppose further that $-\nabla f(x^*)\in {\rm ri}\, \partial g(x^*)$. If  there exists some $c>0$ such that \eqref{tilt2} is satisfied, then $\partial F$ is strongly metrically subregular at $x^*$  for $0$ with any modulus $\kk>c^{-1}$.
\end{corollary}
\begin{proof} Suppose that \eqref{tilt2} holds. Since $g$ is a convex function, its subgradient mapping $\partial g$ is maximal monotone. It follows from \cite[Theorem~2.1]{PR2} and \cite[Theorem~5.3]{LZ} that $\nabla^2_\mathcal{M} g(x^*)$ is positive semidefinite on $T_\mathcal{M}$ in the sense that
\[
\la \nabla^2_\mathcal{M} g(x^*)u,u\ra\ge 0\quad \mbox{for all}\quad u\in T_\mathcal{M}(x^*).
\]
This together with \eqref{tilt2} verifies \eqref{tilt1} and  that $\mu$ in \eqref{c} is smaller than or equal to $c$. Thus $\partial F$ is strongly  metric subregular at $x^*$ for $0$ with any radius $\kk>c^{-1}$. The proof is complete.
\end{proof}
\begin{Remark}{\rm Since $f$ is convex, $\nabla^2 f(x^*)\succeq 0$. The condition \eqref{tilt2} is indeed equivalent to the following one
\begin{equation}\label{tilt3}
{\rm Ker}\, \left(\nabla^2 f(x^*)\right)\cap T_\mathcal{M}(x^*)=\{0\}.
\end{equation}
In some particular application, e.g., $g(x)=\|x\|_1$, the covariant Hessian $\nabla_{\mathcal{M}}g$ is zero, and thus condition \eqref{tilt1} is the same with \eqref{tilt2} (or \eqref{tilt3}). However, in general $\nabla_{\mathcal{M}}g$ may be different from $0$ and thus \eqref{tilt1} is strictly weaker than both \eqref{tilt2} and \eqref{tilt3}. This observation together with the above result and  Proposition~2.3 tells us that \eqref{tilt3}  guarantees the validity of second-order growth condition in \eqref{grow2}, which is part of \cite[Proposition~4.1]{LFP2}.}
\end{Remark}

The following result motivated from \cite[Theorem~3.1]{LFP} also provides  Q-linear convergence of FBS method under the partial smoothness and a positive definite condition. However, as discussed above, our condition \eqref{tilt1} is weaker.

\begin{corollary}[Q-linear convergence under partial smoothness]\label{LCPS} Let $(x^k)_{k\in\NN}$  and $(\al_k)_{k\in\NN}$ be the sequences generated from FBS method. Suppose that the solution set $S^*$ is not empty, $(x^k)_{k\in\NN}$ is converging to some $x^*\in S^*$, $f$ is $\mathcal{C}^2$ around $x^*$, and that $g$ is  $\mathcal{C}^2$-partly smooth at the point $x^*$ relative to the $\mathcal{C}^2$ manifold $\mathcal{M}$. Suppose further that $-\nabla f(x^*)\in {\rm ri}\, \partial g(x^*)$ and the condition \eqref{tilt1} holds for $x^*$. Then there exists some $k\in\NN$ such that
\begin{align*}
\|x^{k+1}-x^*\|&\le\frac{1}{\sqrt{1+\al\kk}}\|x^k-x^*\|\\
|F(x^{k+1})-F(x^*)|&\le \frac{\sqrt{1+\al\kk}+1}{2\sqrt{1+\al\kk}} |F(x^{k})-F(x^*)|
\end{align*}
for any $k>K$, where  $\al$ is any positive  number smaller than $\disp\min\left\{\frac{\al_K}{2},\frac{\theta}{2\lm_{\rm max}(\nabla^2 f(x^*))}\right\}$, $\lm_{\rm max}(\nabla^2 f(x^*))$ is the biggest eigenvalue of     $(\nabla^2 f(x^*))$, and $\kk$ is any positive number smaller than $\mu$ in \eqref{c}.
\end{corollary}
\begin{proof}
Since $f$ is $\mathcal{C}^2$ around $x^*$, $\nabla f$ is locally Lipschitz continuous around $x^*$ with any constant $L$ bigger than $\lm_{\rm max}(\nabla^2 f(x^*))$. Note also from Proposition~\ref{Char1} that condition \eqref{tilt1} ensures that $\partial F$ is strongly metrically subregular at $x^*$ for $0$  with any modulus bigger than $\mu^{-1}$ from \eqref{c}. This together with  Corollary~\ref{TLC2} verifies all the conclusions of this result.
\end{proof}

\subsection{Forward-backward splitting method for $\ell_1$-regularized problems}
In this section we consider the $\ell_1$-regularized optimization problems in \eqref{Prof}. In this case the function $g(x)=\mu\|x\|_1$ belongs to the class of partial smooth functions discussed in Subsection~5.2. However, unlike the study there, we will avoid the nondegeneracy condition $-\nabla f(x^*)\in {\rm ri}\, \partial g(x^*)$ (known also
 as  the strict complementarity condition \cite{HYZ}). To proceed, let us consider the following proposition computing the graphical derivative of $\partial \|\cdot\|_1$.

\begin{proposition}[Graphical derivative of $\partial \mu\|\cdot\|_1$]\label{Gra} Suppose that $\os\in \partial \mu\|x^*\|_1$.  Define $I:=\{j\in \{1,\ldots,n\}\,|\; |\os_j|=\mu\}$, $J:=\{j\in I \,|\;x^*_j\neq 0\}$, $K:=\{j\in I\,|\; x^*_j=0\}$, and $H(x^*):=\{u\in \R^n\,|\; u_j=0, j\notin I\mbox{ and } u_j\os_j\ge 0, j\in K\}$. Then $D\partial \mu\|\cdot\|_1(x^*|\os)(u)$ is nonempty  if and only if $u\in H(x^*)$. Furthermore, we have
\begin{eqnarray}\label{gra}
D\partial \mu\|\cdot\|_1(x^*|\os)(u)=\left\{v\in \R^n\left|\begin{array}{ll} v_j=0, j\in J \\
u_jv_j=0, \os_jv_j\le 0, j\in K
\end{array}\right.\right\}\;\mbox{for all}\;\; u\in H(x^*).\qquad
\end{eqnarray}
\end{proposition}
\begin{proof}
For any $x\in \R^n$, note that
\begin{eqnarray}\label{partial}
\partial \mu\|x\|_1=\left\{s\in \R^n\left|\begin{array}{ll} s_j=\mu\, {\rm sgn}\,(x_j) &\mbox{if}\quad  x_j\neq 0\\
s_j\in[-\mu,\mu]  &\mbox{if}\quad x_j=0
\end{array}\right. \right\},
\end{eqnarray}
where ${\rm sgn}:\R\to \{-1,1\}$ is the {\em sign function}. Take any $v\in D\partial \|\cdot\|_1(x^*|\os)(u)$,  there exists sequence $t^k\dn 0$ and $(u^k,v^k)\to (u,v)$ such that $(x^*,\os)+t^k(u^k,v^k)\in \gph \partial\mu\|\cdot\|_1$.  Let us consider three partitions of $j$ described below:

\noindent {\bf  Partition 1.1:} $j\notin I$, i.e.,  $|\os_j|<\mu$. It follows from \eqref{partial} that  $x^*_j=0$. For sufficiently large $k$, we have $|(\os+t^kv^k)_j|<\mu$ and thus $|(x^*+t^ku^k)_j|=0$ by  \eqref{partial} again. Hence $u^k_j=0$, which implies that $u_j=0$ for all $j\notin I$.

\noindent{\bf Partition  1.2:} $j\in J$, i.e., $|\os_j|=\mu$ and $x^*_j\neq 0$. When $k$ is sufficiently large, we have $(x^*+t^ku^k)_j\neq 0$ and derive from \eqref{partial} that
\[
(\os+t^kv^k)_j=\mu\, {\rm sgn}\, (x^*+t^ku^k)_j=\mu\, {\rm sgn}\,x^*_j=\os_j,
\]
which implies that $v_j=0$ for all $j\in J$.

\noindent{\bf Partition  1.3:} $j\in K$, i.e., $|\os_j|=\mu$ and $x^*_j=0$. If there is a subsequence of $(x^*,\os)_j+t^k(u^k,v^k)_j$ (without relabeling) such that $|(\os+t^kv^k)_j|<\mu=|\os_j|$, we have $\os_jv^k_j< 0$ and $(x^*+t^ku^k)_j=0$ by \eqref{partial}.  It follows that  $u^k_j=0$.  Letting $k\to\infty$,
we have  $u_j=0$ and $\os_j v_j \le 0$. Otherwise, we find some $K>0$ such that $|(\os+t^kv^k)_j|=\mu=|\os_j|$ for all $k>K$, which yields $v^k_j=0$. Taking $k\to\infty$ gives us that $v_j=0$. In both situations, we have $u_jv_j=0$ and $\os_jv_j\le 0$.

Combining the conclusions in three cases above gives us that $u\in H(x^*)$ and also verifies the inclusion ``$\subset$'' in \eqref{gra}. To justify the converse inclusion ``$\supset$'', take $u\in H(x^*)$ and any $v\in \R^n$ with $v_j=0$ for $j\in J$ and  $u_jv_j=0, \os_jv_j\le 0$ for $j\in K$. For any $t^k\dn 0$, we prove that $(x^*,\os)+t^k(u,v)\in \gph \partial\mu\|\cdot\|_1$ and thus verify that $v\in D\partial \mu\|\cdot\|_1(x^*|\os)(u)$. For any $t\in \R$, define the set-valued mapping:
\begin{eqnarray*}
{\rm SGN}(t):=\partial |t|=\left\{\begin{array}{ll} {\rm sgn}\,(t) &\mbox{if}\quad  t\neq 0\\
\ [-1,1] &\mbox{if}\quad t=0.
\end{array}\right.
\end{eqnarray*}
Similarly to the proof of ``$\subset$'' inclusion, we consider three partitions of $j$ as follows:

\noindent{\bf  Partition 2.1:} $j\notin I$, i.e.,  $|\os_j|<\mu$. Since $u\in H(x^*)$, we have $u_j=0$. Note also that $x^*_j=0$. Hence we get $(x^*+t^ku)_j=0$ and $(\os+t^kv)_j\in[-\mu,\mu]$  when $k$ is sufficiently large, which means $(\os+t^kv)_j\in \mu\, {\rm SGN}(x^*+t^ku)_j$.

\noindent{\bf Partition  2.2:} $j\in J$, i.e., $|\os_j|=\mu$ and $x^*_j\neq 0$. Since $v_j=0$, we have
\[
{\rm sgn}\,(\os+t^kv)_j={\rm sgn}\,\os_j={\rm sgn}\,(x^*_j)={\rm sgn}\,(x^*+t^ku)_j
\]
and $(x^*+t^ku)_j\neq 0$ when $k$ is large. It follows that $(\os+t^kv)_j\in\mu\, {\rm SGN}(x^*+t^ku)_j$.

\noindent{\bf Partition  2.3:} $j\in K$, i.e., $|\os_j|=\mu$ and $x^*_j=0$. If $u_j=0$, we have $(x^*+t^ku)_j=0$ and $|(\os+t^kv)_j|\le |\os_j|\le \mu$ for sufficiently large $k$, since $\os_jv_j\le 0$. If $u_j\neq 0$, we have $v_j=0$ and
\[
(\os+t^kv)_j=\os_j={\rm sgn}\,(u_j)={\rm sgn}\,(x^*+t^ku)_j
\]
when $k$ is large, since $u_j\os_j\ge 0$. In both cases, we have $(\os+t^kv)_j\in\mu\, {\rm SGN}(x^*+t^ku)_j$.

From those cases, we always have $(x^*,\os)+t^k(u,v)\in \gph \partial\mu\|\cdot\|_1$ and thus $v\in D\partial \mu\|\cdot\|_1(x^*|\os)(u)$.
\end{proof}

 As a consequence, we establish a characterization of strong metric subregularity for $\partial F_1$.

\begin{theorem}[Characterization of strong metric subregularity for $\partial F_1$]\label{Thm56} Let $x^*$ be an optimal solution to problem~\eqref{Prof}.  Suppose that $\nabla f$ is differentiable at $x^*$. Define  $\mathcal{E}:=\big\{j\in \{1,\ldots,n\}\,|\; |(\nabla f(x^*))_j|=\mu\big\}$, $K:=\{j\in \mathcal{E}\,|\; x^*_j=0\}$, $\mathcal{U}:=\{u\in \R^\mathcal{E}\,|\; u_j(\nabla f(x^*))_j\le 0, j\in K\}$, and $\mathcal{H}_\mathcal{E}(x^*):=[\nabla^2 f(x^*)_{i,j}]_{i,j\in\mathcal{E}}$. Then the following statements are equivalent:
\begin{enumerate}[{\rm(i)}]

\item  The subdifferential mapping  $\partial F_1$ is strongly metrically regular at $x^*$ for $0$

\item  $\mathcal{H}_\mathcal{E}(x^*)$ is positive definite over  $\mathcal{U}$ in the sense that
\begin{equation}\label{HE}
\la \mathcal{H}_\mathcal{E}(x^*) u,u\ra>0\qquad \mbox{for all}\quad u\in \mathcal{U}\setminus\{0\},
\end{equation}

\item  $\mathcal{H}_\mathcal{E}(x^*)$ is nonsingular over $\mathcal{U}$ in the sense that
\begin{equation}\label{HE2}
\ker \mathcal{H}_\mathcal{E}(x^*)\cap \mathcal{U}=\{0\}.
\end{equation}
\end{enumerate}

 Moreover, if \eqref{HE} is satisfied then $\partial F_1$ is strongly metrically regular at $x^*$ for $0$ with any modulus $\kk>c^{-1}$, where
\begin{equation}\label{c3}
c:=\min\left\{\frac{\la \mathcal{H}_\mathcal{E}(x^*)u,u\ra}{\|u\|^2}\,\Big|\; u\in \mathcal{U}\right\}
\end{equation}
with the convention $\frac{0}{0}=\infty$.
\end{theorem}
\begin{proof} First let us verify the equivalence between  (i) and  (ii). Suppose that  (i) is valid, i.e.,  $\partial F_1$ is strongly metrically subregular at $x^*$  for $0$. It follows from Proposition~\ref{SR-grow2} that
there is some $c_1>0$ satisfying that
\begin{equation}\label{vu}
\la w,u\ra\ge c_1\|u\|^2\quad \mbox{for all}\quad w\in D(\nabla f+\partial \mu\|\cdot\|_1)(x^*|0)(u).
\end{equation}
Due to the sum rule of graphical derivative \cite[Proposition~4A.2]{DR}, we have
\[
D(\nabla f+\partial \mu\|\cdot\|_1)(x^*|0)(u)=\nabla^2f(x^*)u+D\partial \mu\|\cdot\|_1(x^*|-\nabla f(x^*))(u).
\]
Thus \eqref{vu} is equivalent to
\begin{equation}\label{vu2}
\la \nabla^2f(x^*)u,u\ra+\la v,u\ra\ge  c_1\|u\|^2\quad \mbox{for all}\quad v\in D\partial \mu\|\cdot\|_1(x^*|-\nabla f(x^*))(u).
\end{equation}
Define $\mathcal{V}:=\{u\in \R^n|\; u_j=0, j\notin\mathcal{E},u_j(\nabla f(x^*))_j\le 0, j\in K\}$. Thanks to Proposition~\ref{Gra}, we have
\begin{equation}\label{zero}
\la v,u\ra=0\qquad \mbox{ for all }\qquad v\in D\partial \mu\|\cdot\|_1(x^*|-\nabla f(x^*))(u), u\in \mathcal{V}.
\end{equation} This together with \eqref{vu2} clearly verifies \eqref{HE}.

Conversely, suppose that \eqref{HE} is satisfied meaning that $c$ in \eqref{c3} is positive. Hence we have
\begin{equation}\label{hu}
\la \mathcal{H}_\mathcal{E}(x^*) u,u\ra\ge c\|u\|^2\qquad \mbox{for all}\quad u\in \mathcal{U}.
\end{equation}
It follows from Proposition~\ref{Gra} that $D\partial \mu \|\cdot\|_1(x^*|-\nabla f(x^*))(u)\neq \emptyset$ if and only if $u\in \mathcal{V}$. For any $u\in \mathcal{V}$ and $v\in D\partial \mu\|\cdot\|_1(x^*|-\nabla f(x^*))(u)$, we write $u=\{0_{\mathcal{E}^C}\}\times\{u_\mathcal{E}\}\in \R^{\mathcal{E}^C}\times\R^\mathcal{E}$ with $\mathcal{E}^C=\{1,\ldots,n\}\setminus\mathcal{E}$. Obtain from  \eqref{HE} and \eqref{zero}  that
\[
\la \nabla^2f(x^*)u,u\ra+\la v,u\ra=\la \mathcal{H}_\mathcal{E}(x^*) u_\mathcal{E},u_\mathcal{E}\ra+\la v,u\ra\ge c\|u_\mathcal{E}\|^2+0=c\|u\|^2,
\]
which verifies \eqref{vu2} and thus \eqref{vu} with $c_1=c$. By Proposition~\ref{SR-grow2}, $\partial F_1$ is strongly metrically subregular with any modulus $\kk>c^{-1}$.  This clarifies the equivalence between  (i) and  (ii) and the last statement of the theorem. Moreover, the equivalence between (ii) and (iii) is trivial due to the fact that $f$ is convex and thus $\mathcal{H}_\mathcal{E}(x^*)$ is positive semi-definite. \end{proof}

\begin{corollary}[Linear convergence of FBS method for $\ell_1$-regularized problems] Let $(x^k)_{k\in\NN}$ and $(\al_k)_{k\in\NN}$ be the sequences generated from FBS method for problem \eqref{Prof}. Suppose that the solution set $S^*$ is not empty, $(x^k)_{k\in\NN}$ is converging to some $x^*\in S^*$, and that  $f$ is $\mathcal{C}^2$ around $x^*$. If condition \eqref{HE} holds, then $(x^k)_{k\in\NN}$ and $(F_1(x^k))_{k\in\NN}$ are Q-linearly convergent to $x^*$ and $F_1(x^*)$  respectively with  rates determined in {\rm Corollary~\ref{LCPS}}, where $\kk$ is any positive number smaller than $c$ in \eqref{c3}.
\end{corollary}
\begin{proof}
The result follows from Corollary~\ref{TLC2}, Proposition~\ref{Gra}, and the proof of Corollary~\ref{LCPS}.
\end{proof}

\begin{Remark}{\rm  It is worth noting  that condition \eqref{HE2} is strictly weaker than the assumption used in \cite{HYZ} that $H_{\mathcal{E}}$ has full rank  to obtain the linear convergence of FBS for \eqref{Prof}. Indeed, let us take into account the case $n=2$, $\mu=1$, and  $f(x_1,x_2)=\frac{1}{2}(x_1+x_2)^2+x_1+x_2$. Note that $x^*=(0,0)$ is an optimal solution to problem \eqref{Prof}. Moreover, direct computation gives us that  $\nabla f(x^*)=(1,1)$, $\mathcal{E}=\{1,2\}$, $\mathcal{V}=\R_-\times\R_-$, and $H_\mathcal{E}(x^*)=\begin{pmatrix} 1 & 1\\1 &1\end{pmatrix}$. It is clear that $H_\mathcal{E}(x^*)$ does not have full rank, but condition  \eqref{HE} and its equivalence \eqref{HE2} hold.
}
\end{Remark}

\subsection{Global Q-linear convergence of ISTA on Lasso problem}
In this section we study the linear convergence of ISTA for Lasso problem \eqref{Pros}. The following lemma taken from \cite[Lemma~10]{BNPS} plays an important role in our proof.

\begin{lemma}[Global error bound]\label{eb} Fix any $R>\frac{\|b\|^2}{2\mu}$. Suppose that $x^*$ is an optimal solution to problem \eqref{Pros}. Then we have
\begin{equation}
F_2(x)-F_2(x^*)\ge \frac{\gm_R}{2} d^2(x;S^*) \quad \mbox{for all}\quad \|x\|_1\le R,
\end{equation}
where
\begin{equation}\label{gamma_R}
\gm_R:=\nu^2\left(1+\frac{\sqrt{5}}{2}\mu R+(R\|A\|+\|b\|)(4R\|A\|+\|b\|\right)^{-1}
\end{equation}
while $\nu$ is the Hoffman constant defined in {\rm \cite[Definition~1]{BNPS}} only depending on the initial data $A, b, \mu$.
\end{lemma}

\begin{theorem}[Global Q-linear convergence of ISTA]\label{ISTA} Let $(x^k)_{k\in\NN}$ be the sequence generated by ISTA for problem~\eqref{Pros} that converges to an optimal solution $x^*\in S^*$. Then $(x^k)_{k\in\NN}$ and $(F_2(x^k))_{k\in\NN}$ are  globally Q-linearly convergent to $x^*$ and $F_2(x^*)$ respectively:\begin{align}
&\|x^{k+1}-x^*\|\le \frac{1}{\sqrt{1+\frac{\al\gm_R}{4}}}\|x^{k}-x^*\|\label{Con1}\\
&|F_2(x^{k+1})-F_2(x^*)|\le \frac{2\sqrt{1+\frac{\al\gm_R}{4}}}{\sqrt{1+\frac{\al\gm_R}{4}}+1}|F_2(x^k)-F_2(x^*)|\label{Con2}
\end{align}
for all $k\in\NN$, where $R$  is any number bigger than $\|x^0\|+\frac{\|b\|^2}{\mu}$  and
$\gamma_R$ is given as in \eqref{gamma_R} while $\al:=\frac{1}{2} \min\left\{\sigma,\frac{\theta}{\lambda_{\rm max}(A^TA)}\right\}$.
\end{theorem}
\begin{proof}
Note that Lasso always has optimal solutions. With  $x^*\in S^*$, we have
\[
F_2(0)=\frac{1}{2}\|b\|^2\ge F_2(x^*)\ge \mu\|x^*\|_1,
\]
which implies that $\|x^*\|\le \|x^*\|_1\le \frac{1}{2\mu}\|b\|^2$. It follows from Corollary~\ref{prop1}(i) that
\[
\|x^k\|\le \|x^k-x^*\|+\|x^*\|\le \|x^0-x^*\|+\|x^*\|\le \|x^0\|+2\|x^*\|\le \|x^0\|+\frac{\|b\|^2}{\mu}<R
\]
for all $k\in \NN$. Thanks to Lemma~\ref{eb}, Corollary~\ref{prop1}(i), and Proposition~\ref{lema-alpha} we have
\begin{equation}\label{min}
\|x^k-x^*\|^2-\|x^{k+1}-x^*\|^2\ge \al\gm_R d^2(x^{k+1};S^*)
\end{equation}
with $\al=\frac{1}{2} \min\left\{\sigma,\frac{\theta}{\lambda_{\rm max}(A^TA)}\right\}$ and the note that $\lambda_{\rm max}(A^TA)$ is the global Lipschitz constant of $\frac{1}{2}\|Ax-b\|^2$. The proof of \eqref{Con1} and \eqref{Con2} are quite similar to the one of \eqref{LC4} and \eqref{LC5} in Theorem~\ref{TLC3} by using \eqref{min} instead of \eqref{im1} there.
\end{proof}

By using Lemma~\ref{eb} and Theorem~\ref{TLC1}, we can prove that indeed $(x^k)_{k\in\NN}$ and $(F_2(x^k))_{k\in\NN}$ are  converging globally R-linearly to $x^*$ and $F_2(x^*)$ with better rates $(1+\al\gm_R)^{-\frac{1}{2}}$ and $(1+\al\gm_R)^{-1}$, respectively. A very similar argument has been obtained recently from \cite{BNPS} with a different approach via K\L-inequality. Here we prove the global Q-linear convergence.  Observe further that the linear rates in Theorem~\ref{ISTA} depends on the initial point $x^0$. However, the local linear rates around optimal solutions are uniform and independent from the choice of $x^0$ as mentioned in Remark~\ref{rem}.

\begin{corollary}[Local Q-linear convergence of ISTA with uniform rate]
Let $(x^k)_{k\in\NN}$ be the sequence generated by ISTA for problem \eqref{Pros} that converges to an optimal solution $x^*\in S^*$. Then \eqref{Con1} and \eqref{Con2} are satisfied when $k$ is sufficiently large, where $\al= \min\left\{\frac{\sigma}{2},\frac{\theta}{2\lambda_{\rm max}(A^TA)}\right\}$ and $R$ is any number bigger than $\frac{\|b\|^2}{2\mu}$.
\end{corollary}
\begin{proof}
Since $x^k$ is converging to $x^*\in S^*$. It follows from the proof of Theorem~\ref{ISTA} that $\|x^*\|\le \frac{\|b\|^2}{2\mu}<R$. Hence there exists  $K\in\NN$ such that $\|x^k\|<R$ for any $k>K$. By using  Lemma~\ref{eb} and Corollary~\ref{prop1}(i), we also obtain  \eqref{min} for all $k>K$.  Following the same arguments as in Theorem~\ref{ISTA} justifies the corollary.
\end{proof}

\section{Uniqueness of optimal solution to $\ell_1$-regularized least square optimization problems}
As discussed in Section 1, the linear convergence of ISTA for Lasso was sometimes obtained by imposing an additional  assumption that Lasso has a unique optimal  solution $x^*$; see, e.g., \cite{TBZ}. Since $\partial F_2$ is always metrically subregular at $x^*$ for $0$ from Remark~\ref{rem}, the uniqueness of $x^*$ is equivalent to the strong metric subregularity of $\partial F_2$ at $x^*$ for $0$. This observation together with  Theorem~\ref{Thm56} allows us to characterize the uniqueness of optimal solution to Lasso in the below theorem. A different characterization for this property could be found in \cite[Theorem~2.1]{ZYC}. Suppose that $x^*$ is an optimal solution, which means $-A^T(Ax^*-b)\in \mu\partial \|x^*\|_1$. In the spirit of Proposition~\ref{Thm56} with $f(x)=\frac{1}{2}\|Ax-b\|^2$, define
\begin{eqnarray}\label{80}
\mathcal{E}:=\big\{j\in \{1,\ldots,n\}\,\big |\; |(A^T(Ax^*-b))_j|=\mu\big\},\; K :=\{j\in \mathcal{E}\,|\; x^*_j=0\},\;  J:=\mathcal{E}\setminus K.\qquad
\end{eqnarray}
Since $-A^T(Ax^*-b)\in\partial  \mu\|x^*\|_1$, if $x_j^*\neq 0$ then $(A^T(Ax^*-b))_j=-\mu\,{\rm sign} (x^*_j)$. This tells us that $J=\{j\in \{1,\ldots,n\}|\; x^*_j\neq 0\}:={\rm supp}\, (x^*)$. Furthermore, given an index set $I\subset \{1,\ldots, n\}$, we denote $A_I$ by the submatrix of $A$ formed by its columns $A_i$, $i\in I$ and  $x_I$ by the subvector of $x\in \R^n$ formed by $x_i$, $i\in I$. For any $x\in \R^n$, we also define $\sign (x):= (\sign(x_1), \ldots, \sign(x_n))^T$  and ${\rm Diag}\,(x)$ by the square diagonal matrix with the main entries $x_1, x_2, \ldots, x_n$.

\begin{theorem}[Uniqueness of optimal solution to Lasso problem]\label{uni} Let $x^*$ be an optimal solution to problem \eqref{Pros}. The following statements are equivalent:
\begin{enumerate}[{\rm (i)}]

\item  $x^*$ is  the unique optimal solution to Lasso \eqref{Pros}.

\item The system $A_Jx_J-A_KQ_Kx_K=0$ and  $x_K\in \R^K_+$ has a unique solution $(x_J,x_K)=(0_J,0_K)\in \R^J\times \R^K$, where $Q_K:={\rm Diag}\,\big[\sign(A_K^T(A_Jx^*_J-b))\big]$.

\item   The submatrix $A_J$ has full column rank and the columns of  $A_JA_J^\dag A_KQ_K-A_KQ_K$ are {\em positively linearly independent} in the sense that
\begin{equation}\label{ker}
{\rm Ker}\, (A_JA_J^\dag A_KQ_K-A_KQ_K)\cap \R^K_+= \{0_K\},
\end{equation}
 where $A_J^\dag:=(A_J^TA_J)^{-1}A_J^T$ is the Moore-Penrose pseudoinverse of $A_J$.

\item  The submatrix $A_J$ has full column rank and there exists a {\em Slater point} $y\in \R^m$ such that
\begin{equation}\label{rag}
(Q_KA_K^T A_JA_J^\dag-Q_KA_K^T)y<0.
\end{equation}
\end{enumerate}
\end{theorem}
\begin{proof}{\rm
Since $\partial F_2$ is always metrically subregular at $x^*$ for $0$ from Remark~\ref{rem},  (i) means that $\partial F_2$ is strongly metrically subregular at $x^*$ for $0$. Thus, by Theorem~\ref{Thm56},  (i) is equivalent to
\begin{equation}\label{neh}
\la\mathcal{H}_{\mathcal E}u,u\ra>0\quad \mbox{for all}\quad u\in \mathcal{U}\setminus\{0\}
\end{equation}
with $f(x)=\frac{1}{2}\|Ax-b\|^2$ and $ \mathcal U=\{u\in \R^{\mathcal{E}}|\; u_j(\nabla f(x^*))_j\le 0, j\in K\}$. Note that $\mathcal{H}_{\mathcal E}=[\nabla^2 f(x^*)_{i,j}]_{i,j\in \mathcal{E}}=[(A^TA)_{i,j}]_{i,j\in \mathcal{E}}=A_\mathcal{E}^TA_\mathcal{E}$. Hence \eqref{neh} means the system
\begin{equation}\label{neh1}
0=A_\mathcal{E}u=A_Ju_J+A_Ku_K\quad \mbox{and} \quad u_K\in   \mathcal U_K
\end{equation}
has a unique solution $u=(u_J, u_K)=(0_J, 0_K)\in \R^J\times \R^K$, where $ \mathcal U_K$ is defined by
\[
 \mathcal U_K:=\{u\in \R^K|\; u_k(A^T(Ax^*-b))_k\le 0, k\in K\}.
\]
As observed after \eqref{80}, $J={\rm supp}\, (x^*)$, for each $k\in K$ we have
\[
(A^T(Ax^*-b))_k= (A^T(A_Jx^*_J-b))_k=(A^T_K(A_Jx^*_J-b))_k.
\]

It follows that   $ \mathcal U_K=-Q_K(\R^K_+)$ and $Q_K$ is a nonsingular diagonal square matrix (each diagonal entry is either $1$ or $-1$). Uniqueness of system \eqref{neh1} is equivalent to  (ii). This verifies the equivalence between  (i) and  (ii).

Let us justify the equivalence between  (ii) and (iii).  To proceed, suppose that  (ii) is valid, i.e., the system
\begin{equation}\label{AQ}
A_Jx_J-A_KQ_Kx_K=0 \quad \mbox{with}\quad   (x_J,x_K)\in\R^J\times\R^K_+.
\end{equation}
has a unique solution $(0_J,0_K)\in \R^J\times \R^K$. Choose $x_K=0_K$, the latter tells us that equation  $A_Jx_J=0$ has a unique solution $x_J=0$, i.e., $A_J$ has full column rank. Thus $A_J^TA_J$ is nonsingular.   Furthermore, it follows from \eqref{AQ} that $A_J^TA_Jx_J=A_J^TA_KQ_Kx_K$, which means
\begin{equation}\label{AQ1}
x_J=(A_J^TA_J)^{-1}A_J^TA_KQ_Kx_K=A_J^\dag A_KQ_Kx_K.
\end{equation}
This together with \eqref{AQ} tells us that the system
\begin{equation}\label{AQ2}
A_JA_J^\dag A_KQ_Kx_K-A_KQ_Kx_K=(A_JA_J^\dag A_KQ_K-A_KQ_K)x_K=0, x_K\in \R^K_+
\end{equation}
has a unique solution $x_K=0_K\in\R^K$, which clearly verifies \eqref{ker} and thus {(iii)}.

To justify the converse implication, suppose that (iii) is valid. Consider the equation \eqref{AQ} in  (ii), since $A_J$ has the full rank column, we also have \eqref{AQ1}. It is similar to the above justification that  $x_K$ satisfies equation \eqref{AQ2}. Thanks to \eqref{ker} in  (iii), we get from \eqref{AQ2} that $x_K=0_K$ and thus $x_J=0_J$ by \eqref{AQ1}. This verifies that the equation \eqref{AQ} in (ii) has a unique solution $(x_J,x_K)=(0_J,0_K)$.

Finally, the equivalence between  (iii) and  (iv) follows from the well-known Gordan's lemma and the fact that the matrix $A_JA_J^\dag$ is symmetric.}
\end{proof}

Next let us discuss some known conditions relating the uniqueness of optimal solution to Lasso.
In \cite{F}, Fuchs introduced a sufficient condition for the above property:
\begin{eqnarray}
&&A_J^T(A_Jx^*_J-b)=-\mu\,\sign(x^*_J),\label{opts}\\
&&\|A^T_{J^c}(A_Jx^*_J-b)\|_\infty<\mu,\label{emp}\\
 &&\mbox{$A_J$ has full column rank.}\label{xu}
\end{eqnarray}
The first equality \eqref{opts}  indeed tells us that $x^*$ is an optimal solution to Lasso problem. Inequality \eqref{emp} means that $\mathcal{E}=J$, i.e., $K=\emptyset$  in Theorem~\ref{uni}. \eqref{xu} is also present in our  characterizations. Hence Fuchs' condition implies  (iii) in Theorem~\ref{uni} and is clearly not a necessary condition for  the uniqueness of optimal solution to Lasso problem, since in many situations the set $K$ is not empty.

Furthermore, in  the recent work \cite{Ti} Tibshirani shows that the optimal solution $x^*$ to problem \eqref{Pros} is unique when the matrix $A_\mathcal{E}$ has full column rank. This condition is sufficient for our  (ii) in Theorem~\ref{uni}. Indeed, if $(x_J,x_K)$ satisfies system \eqref{AQ} in  (ii), we have $A_\mathcal{E}[x_J\,\, -Q_Kx_K]^T=0$, which implies that $x_J=0$ and $Q_Kx_K=0$ when $\ker A_\mathcal{E}=0$. Since $Q_K$ is invertible, the latter tells us that $x_J=0$ and $x_K=0$, which clearly verifies  (ii). Tibshirani's condition is also necessary for the uniqueness of optimal solution to Lasso problem for {\em almost all} $b$ in \eqref{Pros}, but it is not for any $b$; a concrete example could be found in \cite{ZYC}.

In the recent works \cite{ZYC,ZYY},  the following useful characterization of unique solution to Lasso has been established  under mild assumptions:
\begin{eqnarray}
&&\mbox{There exists $y\in \R^m$ satisfying $A^T_Jy={\rm sign}\, (x^*_J)$ and $\|A^T_{K}y\|_\infty<1$},\label{exi}\\
 &&\mbox{$A_J$ has full column rank.}\nonumber
\end{eqnarray}
It is still open to us to connect directly this condition to those ones in Theorem~\ref{uni},  although they must be logically equivalent under
the assumptions required in \cite{ZYC,ZYY}. However, our approach via second-order variational analysis is completely different and also provides several new characterizations for the uniqueness of optimal solution to Lasso. It is also worth mentioning here that the standing assumption in \cite{ZYC} that $A$ has full row rank is  relaxed in our study.

To end this section, we note that the procedure in this section could be extended to investigate the same property for other structured optimization problem in \cite{ZYY} as well as the well-known {\em nuclear norm  regularized least square} optimization problem
\begin{equation}\label{nuc}
\min_{X\in \R^{p\times q}}\qquad h(X):=\|\mathcal{A}X-b\|^2+\mu\|X\|_*,
\end{equation}
where $\mathcal{A}:\R^{p\times q}\to \R^m$ is a linear operator and $\|X\|_*$
 is the trace norm (known as well the nuclear norm) of $X$. The recent main result in \cite{ZS} could imply that $\partial h$ is locally metrically regular at any point on its graph under a mild assumption. So $X^*$ is a unique optimal solution to \eqref{nuc} if and only if it is an optimal solution and $\partial h$ is strongly metrically regular at $X^*$ for $0$ under a mild assumption. Then the same approach to problem~\eqref{nuc} will lead  us to several new complete  characterizations for the uniqueness of optimal solution to this problem provided that the graphical derivative $D\partial \|\cdot\|_*$ is fully calculated.

\section{Conclusion}

In this paper we analyze the Q-linear convergence of the forward-backward splitting method for solving nonsmooth convex optimization problems and the uniqueness of optimal solution to Lasso. Our work recovers several recent results in \cite{BNPS,DL,F,HYZ,LFP, LFP2,NNG,Ti,TBZ} and reveals many new information. (Strong) Metric subregularity on the subdifferential and second-order growth condition play significant roles in our analysis. It is well-recognized that K\L-inequality with order  $\frac{1}{2}$, which is equivalent to both latter properties in convex frameworks, is a very useful tool to guarantee the convergence of many proximal-type algorithms even for nonconvex optimization problems.  In future research we intend to study the connection of metric subregularity of subdifferential and second-order growth condition with K\L\  inequality and  their effects to the convergence of proximal algorithms in nonconvex settings. Extending the approach in Section 6 to investigate the uniqueness of optimal solution to $\ell_0$-optimization problem is also a potential project that we are  working on.


\begin{thebibliography}{1}

\small


\bibitem{AC} D. Az\'e and J.-N. Corvellec: Nonlinear local error bounds via a change of metric, {\em J. Fixed Point Theory Appl.} {\bf 16} (2014), 251--372.

\bibitem{AG} F. J. Arag\'on Artacho and M. H. Geoffroy, Characterizations of metric regularity of subdifferentials, {\em J. Convex Anal.} {\bf 15}
(2008), 365--380.

\bibitem{AG2} F. J.  Arag\'on Artacho and M. H. Geoffroy: Metric subregularity of the convex subdifferential in Banach spaces, {\em J. Nonlinear Convex Anal.} {\bf 15} (2014),  35--47.


\bibitem{BBT} H.H. Bauschke, J.  Bolte, and M. Teboulle: A descent lemma beyond Lipschitz gradient continuity: first-order methods revisited and applications, {\em Math. Oper. Res.} {\bf 42} (2017), 330--348.


\bibitem{BC} H. H. Bauschke and P. L. Combettes: {\it Convex Analysis and Monotone Operator Theory in Hilbert Spaces}. Springer, New York (2011).

\bibitem{BL} K. Bredies and D. A. Lorenz: Linear convergence of iterative soft-thresholding. {\it Journal of Fourier Analysis and Applications} {\bf 14} (2008), 813--837.

 \bibitem{BNPS} J. Bolte, T.P. Nguyen, J. Peypouquet, and B. W. Suter: From error bounds to the complexity of first-order descent methods for convex functions, {\it Mathematical Programming} , (2016). doi:10.1007/s10107-016-1091-6.

 \bibitem{YN} J. Y. Bello Cruz and T. T. A. Nghia:  On the convergence of the proximal forward-backward splitting method with linesearches, {\it Optim. Method Softw.} {\bf 31} (2016), 1209--1238.

\bibitem{BPN} H. H. Bauschke, H. M. Phan, and D. Noll: Linear and strong convergence of algorithms involving averaged nonexpansive operators, {\em
J. Math. Anal.  Appl.},  {\bf 421} (2015), 1--20.




 \bibitem{Bon12}
S.~Bonettini and V.~Ruggiero,
On the convergence of primal-dual hybrid gradient algorithms for total variation image restoration.
{\it J. Math. Imaging Vision}, {\bf 44} (2012) 236--253.



\bibitem{BTe} A. Beck and M. Teboulle: {\it Gradient-Based Algorithms with Applications to Signal Recovery
Problems}. in {\it Convex Optimization in Signal Processing and
Communications}, (D. Palomar and Y. Eldar, eds.)  42--88 University
Press, Cambribge (2010).




 \bibitem{C} I. Csisz\'ar: Why least squares and maximum entropy? An axiomatic approach to inference for linear inverse problems, {\em Ann. Statist.} {\bf 19} (1991),  2032--2066.

\bibitem{CP}  P. L. Combettes and J.-C. Pesquet: Proximal splitting methods in signal processing. in {\it
Fixed-Point Algorithms for Inverse Problems. Science and
Engineering}. {\it Springer Optimization and Its Applications} {\bf
49} (2011), 185--212 Springer, New York.

\bibitem{CW} P. L. Combettes and V. R. Wajs: Signal recovery by proximal forward-backward splitting, {\it Multiscale Model. Simul.} {\bf
4} (2005), 1168--1200.

\bibitem{dau} I. Daubechies, M. Defrise, and D. De Mol: An iterative
thresholding algorithm for linear inverse problems with a sparsity
constraint. {\it Comm. Pure Appl. Math.}
{\bf 57} (2004), 1413--1457.

\bibitem{DR}  A. L. Dontchev and  R. T. Rockafellar: {\em Implicit Functions and Solution Mappings. A View from Variational Analysis},
Springer, Dordrecht, 2009.

\bibitem{DL} D. Drusvyatskiy and A. Lewis: Error bounds, quadratic growth, and linear convergence of proximal methods, {\em Math. Oper. Res.} doi.org/10.1287/moor.2017.0889


\bibitem{DMN} D. Drusvyatskiy, B. S. Mordukhovich and T. T. A. Nghia, Second-order growth, tilt stability, and metric regularity of
the subdifferential, {\em J. Convex Anal.} {\bf 21} (2014), 1165--1192.





\bibitem{F} J-J. Fuchs: On sparse representations in arbitrary redundant bases. {\it IEEE Trans. Inform. Theory},  {\bf 50} (2004), 1341--1344.

\bibitem{H} A. J. Hoffman: On approximate solutions of systems of linear inequalities. {\it J. Res. Nat. Bur. Standards}, {\bf 49} (1952), 263--265.

\bibitem{HYZ} E. T. Hale, W. Yin, and Y. Zhang: Fixed-point continuation for $\ell_1$-minimization: methodology and convergence. {\it SIAM J. Optim.} {\bf 19} (2008), 1107--1130.

\bibitem{LFP} J. Liang, J. Fadili, and G. Peyr\'e: Local linear convergence of forward-backward under partial smoothness, {\em Adv. Neural Inf. Process Syst.} (2014).

\bibitem{LFP2} J. Liang, J. Fadili, and G. Peyr\'e: Activity identification and local linear convergence of forward--backward type methods, {\it SIAM J. Optim.} {\bf 27} (2017), 408--437.


\bibitem{L} A. S. Lewis: Active sets, nonsmoothness, and sensitivity, {\it SIAM J. Optim.} {\bf 23} (2002), 702--725.

\bibitem{LZ} A. S. Lewis and S. Zhang: Partial smoothness, tilt stability, and generalized Hessians, {\it SIAM J. Optim.} {\bf 23} (2013), 74--94.


\bibitem{LP} G. Li and T.K. Pong: Calculus of the exponent of Kurdyka-\L ojasiewicz inequality and its applications to linear convergence of first-order methods, {\em Found. Comp. Math.} (2018) doi.org/10.1007/s10208-017-9366-8

\bibitem{LT} Z.-Q. Luo and P. Tseng: Error bounds and convergence analysis of feasible descent methods: a general approach, {\em Ann. Oper. Res.} {\bf 46} (1993), 157--178.

\bibitem{MM} S. A. Miller and J. Malick: Newton methods for nonsmooth convex minimization: connections among U-Lagrangian, Riemannian Newton and SQP methods,  {\em Math. Program.} {\bf 104} (2005), 609--633.

\bibitem{M1} B. S. Mordukhovich: {\em Variational Analysis and Generalized Differentiation, I: Basic Theory, II: Applications}, Springer,
Berlin (2006).

\bibitem{MN} B. S. Mordukhovich and T. T. A. Nghia: Second-order variational analysis and characterizations of tilt-stable optimal solutions in infinite-dimensional spaces, {\em Nonlinear Anal.} {\bf 86} (2013), 159--180.

\bibitem{MN1} B. S. Mordukhovich and T. T. A. Nghia: Second-order characterizations of tilt stability with applications to nonlinear programming,  {\em Math. Program.} {\bf 149} (2015), 83--104.

\bibitem{NNG} I. Necoara, Yu. Nesterov, and F. Glineur: Linear convergence of first order methods for non-strongly convex optimization, {\em Math. Program.} (2018) doi.org/10.1007/s10107-018-1232-1.

\bibitem{NB} P. Neal and S. Boyd: Proximal Algorithms, {\it Foundations and Trends in Optimization} {\bf 1} (2014), 127--239.



\bibitem{PR2} R. A. Poliquin and R. T. Rockafellar: Tilt stability of a local minimum, {\em SIAM J. Optim.} {\bf 8} (1998), 287--299.



\bibitem{R} S. M. Robinson: Some continuity properties of polyhedral multifunctions, {\em Math. Program. Study} {\bf 14} (1981), 206--214.

\bibitem{rw}  R. T. Rockafellar and R. J-B. Wets: {\em Variational Analysis}, Springer, Berlin, 1998.



\bibitem{S17} S.~Salzo: The variable metric forward-backward splitting algorithm under mild differentiability assumptions, {\it SIAM J. Optim.}, {\bf 27} (2017), 2153--2181.

\bibitem{T} R. Tibshirani: Regression shrinkage and selection via the Lasso, {\em
J. R. Stat. Soc.} {\bf 58} (1996),  267--288

\bibitem{Ti} R. J. Tibshirani: The Lasso problem and uniqueness, {\em Electron. J. Stat.} {\bf  7} (2013), 1456--1490.

\bibitem{Tr} J. Tropp: Just relax: Convex programming methods for identifying sparse
signals in noise, {\it  IEEE Trans. Inform. Theory}, {\bf 52} (2006),1030--1051.

\bibitem{tseng} P. Tseng: A modified forward-backward splitting method for maximal monotone mappings, {\it SIAM J. Control Optim.\/}
{\bf 38} (2000), 431--446.

\bibitem{TBZ} S. Tao, D. Boley, and S. Zhang: Local linear convergence of ISTA and FISTA on the Lasso problem, {\it SIAM J. Optim.}, {\bf 26} (2016), 313--336.



\bibitem{VSK} Y. Vardi, L.A.  Shepp, L. Kaufman: A statistical model for positron emission tomography,  {\em J. Amer. Statist. Assoc.} {\bf 80} (1985),  8--37.



\bibitem{Wa} M. J. Wainwright: Sharp thresholds for high-dimensional and noisy
sparsity recovery using $\ell_1$-constrained quadratic programming (lasso), {\it IEEE Trans. Inform. Theory}, {\bf  55} (2009), 2183--2202.

 \bibitem{ZS} Z. Zhou and A. M-C. So: A unified approach to error bounds for structured convex optimization, {\it Math. Program.}, {\bf 165} (2017), 689--728.


\bibitem{ZYC} H. Zhang, W. Yin, L. Cheng: Necessary and sufficient conditions of solution uniqueness in 1-norm minimization, {\it  J. Optim. Theory Appl.}, {\bf  164} (2015), 109--122.


\bibitem{ZYY} H. Zhang, M. Yan, W. Yin: One condition for solution uniqueness and robustness of both $\ell_1$-synthesis and $\ell_1$-analysis minimizations, {\it Adv. Comput. Math.} {\bf 42} (2016), 1381--1399.





\end{thebibliography}
\end{document}